\documentclass[12pt]{article}

\usepackage{a4,epsfig}
\usepackage{amsfonts,amsmath}
\usepackage{subcaption}
\usepackage{algorithm}
\usepackage{algpseudocode}
\usepackage{pgfplotstable}

\newtheorem{theorem}{Theorem}[section]

\newtheorem{lemma}[theorem]{Lemma}

\newenvironment{proof}{\textbf{Proof.}}{\hfill $\Box$}

\numberwithin{equation}{section} 

\begin{document}

\setcounter{page}{1}

\title{A space-time finite element method for the eddy current
    approximation of rotating electric machines}
\author{Peter~Gangl$^1$, Mario~Gobrial$^2$, Olaf~Steinbach$^2$}
\date{$^1$Johann Radon Institute for Computational and
  Applied Mathematics, \\
Altenberger Stra{\ss}e 69, 4040 Linz, Austria \\[2mm]
$^2$Institut f\"ur Angewandte Mathematik, TU Graz, \\[1mm] 
Steyrergasse 30, 8010 Graz, Austria}

\maketitle

\begin{abstract}
  In this paper we formulate and analyze a space-time finite element
  method for the numerical simulation of rotating electric machines where
  the finite element mesh is fixed in space-time domain.
  Based on the Babu\v{s}ka--Ne\v{c}as theory we prove unique solvability
  both for the continuous variational formulation and for a standard Galerkin
  finite element discretization in the space-time domain. This approach
  allows for an adaptive resolution of the solution both in space and time,
  but it requires the solution of the overall system of algebraic equations.
  While the use of parallel solution algorithms seems to be mandatory,
  this also allows for a parallelization simultaneously in space and time.
  This approach is used for the eddy current approximation of the Maxwell
  equations which results in an elliptic-parabolic interface problem.
  Numerical results for linear and nonlinear constitutive material relations
  confirm the applicability, efficiency and accuracy of the proposed
  approach.
\end{abstract}

\section{Introduction}\label{sec:intro}
Electric machines have become an integral part of everyday life with a
large share of global electric energy being converted into mechanical energy
by electric machines. The efficient and accurate numerical simulation of
electric machines is thus an important topic in particular in order to design
new machines with high performance indicators. Mathematical models for
computing the magnetic flux density and the magnetic field inside an
electric machine are based on low-frequency approximations to Maxwell's
equations such as the magneto-quasi-static or the magneto-static
approximations. While the former accounts for eddy currents in conducting
regions and is also refered to as the eddy current approximation of
Maxwell's equations \cite{Rodrguez2010}, the more widely used magneto-static
approximation ignores these effects. Eddy currents yield thermal losses
\cite{Ida1997} and thus are typically an unwanted effect in electric machines
and therefore often counteracted, e.g., by assembling the machine from thin
laminated steel sheets. Nevertheless, these effects are still present, e.g.,
in permanent magnets or when non-laminated designs are chosen
\cite{Mellak2022}, and their accurate computation is of high relevance.
While the magneto-static approximation to Maxwell's equations for rotating
electric machines results in a sequence of independent static problems, the
eddy current approximation yields a time-dependent problem of mixed
parabolic-elliptic type \cite{Bachinger2005}. This type of problems is
typically solved by frequency domain methods such as the multi-harmonic
finite element method \cite{wolfmayr2023posteriori} or the harmonic
balance method \cite{GyselinckEtAl2003, Putek2019}, or by classical
time-stepping methods in time domain \cite{Thomee2006}.
Since classical time-stepping methods suffer from the curse of sequentiality,
different ways to employ parallelization also in time direction have been
investigated over the past decades \cite{Gander2015} including shooting
methods, domain decomposition or multigrid methods
\cite{GanderNeumueller2016} in time. We mention the application of
parareal \cite{GanderKuchytska2019, Kulchytska2021} and multigrid reduction
in time \cite{Bolten2020, Friedhoff2019} algorithms to the time-parallel
simulation of eddy current problems for electric machines.

On the other hand, space-time finite element methods \cite{StYa19} for
the numerical solution of time-dependent partial differential equations
have gained increasing attention over the past decade due to increasing
computing capabilities. Here, the idea is to treat the time variable like
an additional space variable and to construct a $(d+1)$-dimensional
space-time mesh when the spatial domain is in $\mathbb R^d$. In this setting,
a moving domain can conveniently be captured by the space-time mesh. While,
at the first glance, the method comes with the challenge of higher-dimensional
linear systems to be solved, it allows for both parallelization
\cite{GanderNeumueller2016} and adaptivity
\cite{LangerSchafelner2020, SteinbachYang2018} not only in space or time,
but also in space-time. Moreover, in the context of optimization problems
with partial differential equations as constraint, and involving an adjoint
state which is directed backward in time, space-time methods allow for an
additional level of parallelism by solving the coupled system for the state
and the adjoint in parallel \cite{LSTY2021}. Finally, we will see that for
space-time finite element methods temporal periodicity conditions as they
appear for rotating electric machines at a fixed operating point can be
incorporated in a straightforward manner.

The numerical analysis of space-time variational formulations for parabolic
evolution equations in Bochner spaces is based on the
Babu\v{s}ka--Ne\v{c}as theory \cite{BaAz72,Ne62} which requires the
proof of an inf-sup stability condition to ensure uniqueness, and
of a surjectivity condition to ensure existence of a solution.
In the context of space-time finite element methods this was done
in \cite{SchwabStevenson:2009}, see also
\cite{Andreev,St15,UrbanPatera}. Alternatively, one may use isogeometric
space-time finite element methods \cite{LangerMooreNeumueller2016},
least squares formulations \cite{StevensonWesterdiep:2021}, or
a Galerkin space-time finite element method in anisotropic
Sobolev spaces \cite{SteinbachZank}.

In this paper, we extend the analysis of \cite{St15}, which is based on
the Babu\v ska--Ne\v cas theorem, to the case of a coupled elliptic-parabolic
partial differential equation which is formulated in a spatial domain
which is changing in time. In particular we will restrict ourselves to the
case of a rotating subdomain as it is the case for rotating electric machines.
While we exploit this property in our proof of surjectivity of the bilinear
form (Lemma \ref{Lemma surjective}), we claim that the presented approach
can also be applied in more general settings, also including compressible
deformations of the computational domain.

The rest of this paper is organized as follows:
In Section \ref{sec:model_descr}, starting out from Maxwell's equations,
we derive the mathematical model of two-dimensional magneto-quasi-statics
which we consider in the sequel. The main part of this paper is
Section \ref{sec:var_form} where we verify the conditions of the
Babu\v ska--Ne\v cas theorem and conclude existence of a unique solution.
In Section \ref{sec:space_time_fe_dis} we introduce a space-time finite
element discretization and give the corresponding stability and error
estimates before resorting to numerical experiments in
Section \ref{sec:numerical_experiments}, where we also discuss the
parallel solution of the algebraic equations. In addition to the linear
model problem we also include a nonlinear model to describe the reluctivity
in iron. Finally, we summarize and comment on ongoing and future work.

\section{Model description}\label{sec:model_descr}
To model the electromagnetic fields in a rotating
electric machine, e.g., an electric motor, we consider the eddy current
approximation of the Maxwell equations, see, e.g., \cite{LaPaRe19},
\begin{align}\label{eqn:maxwell}
  \text{curl}_y H(y,t) = J(y,t), \quad
  \text{curl}_y E(y,t) = - \partial_t B(y,t), \quad
  \text{div}_y B(y,t) = 0,
\end{align}
subject to the constitutive equations
\begin{align}
\label{eqn:constitutive}
  B(y,t) = \mu(y)H(y,t) + M(y,t), \quad
  J(y,t) = J_i(y,t) + \sigma(y) \Big[E(y,t) + v(y,t) \times B(y,t) \Big],
\end{align}
where $\mu$ is the material dependent magnetic permeability, $\sigma$ is the
electric conductivity, $J_i$ is an impressed electric current, $M$ is the
magnetization which vanishes outside permanent magnets, and
$v=\frac{d}{dt}y(t)$ is the velocity along the trajectory
$y(t)=\varphi(t,x) \in {\mathbb{R}}^3$ for a reference point
$x \in {\mathbb{R}}^3$.
We assume that the deformation $\varphi$ is bijective and sufficiently
regular for all $t \in (0,T)$, satisfying $\varphi(0,x)=x$.
Here, $T > 0$ is a given time horizon, and we assume that
$\text{div}_y v(y,t) = 0$.

When using the vector potential ansatz $B = \text{curl}_y A$
satisfying $\text{div}_y B = \text{div}_y \text{curl}_y A = 0$,
we can rewrite the second equation in \eqref{eqn:maxwell} as
$0 = \text{curl}_y E + \partial_t B = \text{curl}_y [ E + \partial_t A]$,
which implies, recall that the vector potential $A$ is unique up to a gradient
field only, $E=-\partial_tA$. When using the reluctivity $\nu=1/\mu$ we then
have $H = \nu (B-M) = \nu (\text{curl}_yA-M)$ in order to rewrite the
first equation in \eqref{eqn:maxwell} as
\begin{align}\label{eqn:vector potential pde}
  \text{curl}_y \Big[ \nu(y) \Big(
  \text{curl}_y A(y,t) - M(y,t) \Big) \Big] =
  J_i(y,t) - \sigma(y) \Big[ \partial_t A(y,t) +
  \text{curl}_y A(y,t) \times v(y,t) \Big] .
\end{align}
Assuming that
\[
H(y,t) = (H_1(y_1,y_2,t),H_2(y_1,y_2,t), 0)^\top, \quad
M(y,t) = (M_1(y_1,y_2,t),M_2(y_1,y_2,t), 0)^\top,
\]
and
\[
  v(y,t) = (v_1(y_1, y_2,t),v_2(y_1,y_2,t),0)^\top, \quad
  J_i(y,t) = (0,0,j_i(y_1,y_2,t))^\top,
\]
which is often (approximately) the case for electric machines, it follows
that \linebreak
$A=(0,0,u(y_1,y_2,t))^\top$, and we can consider a spatially
two-dimensional reference domain $\Omega \subset {\mathbb{R}}^2$
describing the cross-section of the electric motor.
Using $x=(x_1,x_2,0)^\top$ for $(x_1,x_2) \in \Omega$, we can rewrite
\eqref{eqn:vector potential pde} as
\begin{eqnarray}\label{eqn:eddy_current_pde}
  && \sigma(y_1,y_2) \frac{d}{dt}u(y_1,y_2,t) - \text{div}_{(y_1,y_2)}
     [\nu(y_1,y_2) \nabla_{(y_1,y_2)} u(y_1,y_2,t)] \\
  && \nonumber \hspace*{4cm} =
  j_i(y_1,y_2,t) - \text{div}_{(y_1,y_2)}[\nu(y_1,y_2) M^\perp(y_1,y_2,t)],
\end{eqnarray}
where
\begin{align*}
  \frac{d}{dt}u(y_1,y_2,t) := \partial_t u(y_1,y_2,t)
  + v(y_1,y_2,t) \cdot \nabla_{(y_1,y_2)} u(y_1,y_2,t)
\end{align*}
is the total time derivative, and
$M^\perp = (-M_2(y_1,y_2,t), M_1(y_1,y_2,t))^\top$ is the
perpendicular of the first two components of the magnetization $M$.
In addition to the partial differential equation
\eqref{eqn:eddy_current_pde} we consider homogeneous
Dirichlet boundary conditions $u=0$ on $\partial \Omega \times (0,T)$
which implies that $B \cdot n = 0$ on $\partial \Omega \times (0,T)$,
i.e., no magnetic flux leaves the computational domain,
and either the initial condition $u(x_1,x_2,0)=0$ or
the periodicity condition $u(x_1,x_2,T)=u(x_1,x_2,0)$, both for $(x_1,x_2)
\in \Omega$ when $\sigma(x_1,x_2) > 0$. Note that, in the case of
periodicity conditions, we assume that also the geometry and the sources
are periodic with respect to the period $T$.

\begin{figure}[ht]
\begin{center}
\includegraphics[width=9cm]{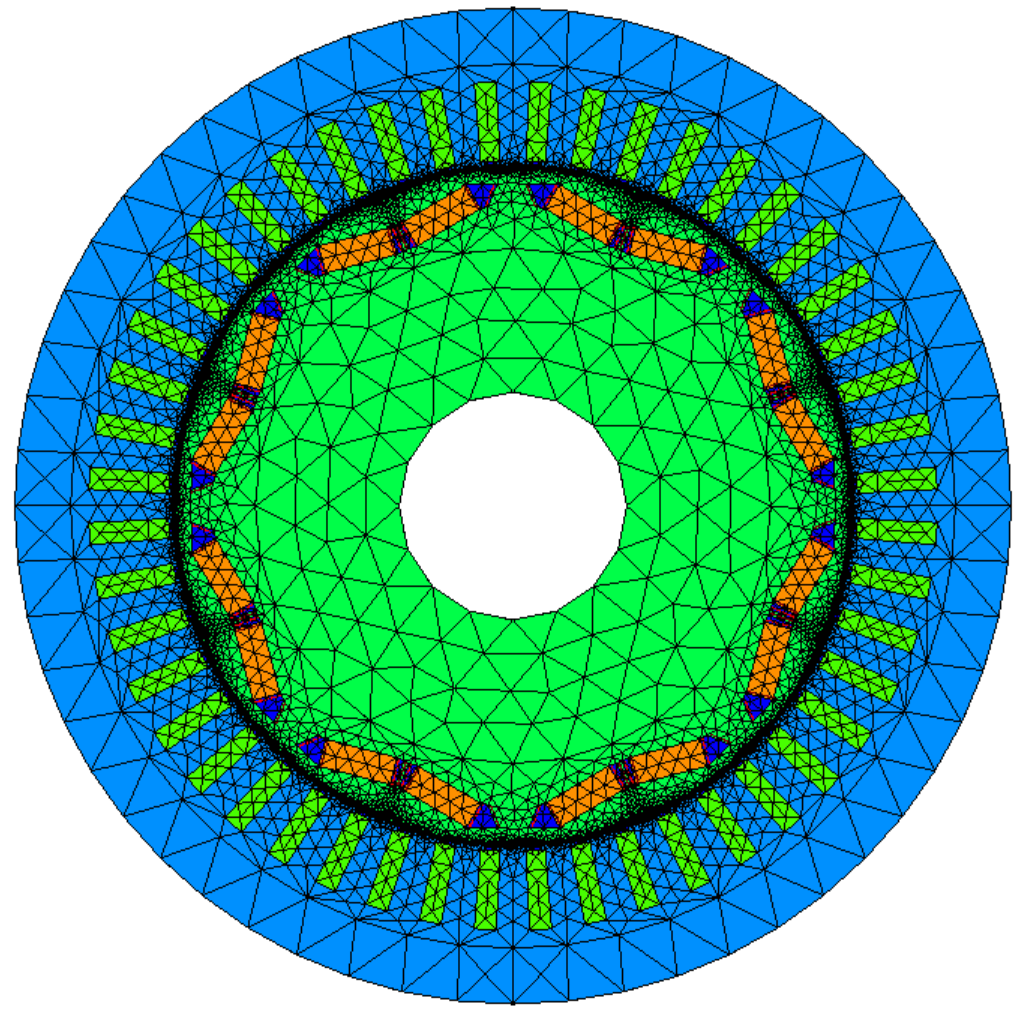}
\end{center}
\caption{Finite element mesh of the reference domain $\Omega$ describing 
an electric motor with the
stator $\Omega_s$ including the coils, the rotor domain $\Omega_r=\Omega_r(t)$
including the magnets and surrounding air pockets, and the thin air gap $\Omega_{air}$ separating $\Omega_r$ from $\Omega_s$.}
\label{Bild Motor}
\end{figure}

We consider an electric motor as shown in Fig.~\ref{Bild Motor} which
consists of a rotor in $\Omega_r(t)$, the stator in $\Omega_s$, and the
air domain $\Omega_{air}$ which is non-conducting, i.e.,
$\sigma = 0$ in $\Omega_{air}$. In this case, the evolution equation
\eqref{eqn:eddy_current_pde} degenerates to a coupled parabolic-elliptic
interface problem. Within the stator there are $48$ coils excited with a
current and which are considered to be non-conducting, since certain
materials are used to ensure this property. We furthermore denote the
union of all non-conducting
regions ($\sigma = 0$) by $\Omega_{non}$, and the regions with conducting material ($\sigma > 0$) by
$\Omega_{con}$. The stator in $\Omega_s$ is
fixed, i.e., $y= \varphi(t,x)=x$ for all $t \in (0,T)$ implying $v \equiv 0$,
but the rotor and the magnets within the rotor are rotating.

To cover all different regions, i.e., rotor, stator, and air, in a unified
framework, we use polar coordinates to write
$(x_1,x_2)^\top = r (\cos \phi , \sin \phi)^\top \in \Omega$
for $\phi \in [0,2\pi)$ and $r \in (r_0,R)$, where $r_0$ and $R$ are
the interior and exterior radii of the motor, respectively. In addition,
for $ r \in (r_0,r_1)$ we describe the rotor in the domain $\Omega_r(t)$
which is rotating with a velocity $\alpha$, and
which may contain non-conducting areas such as air as well, while for
$r \in (r_1,r_2)$ we have the stator domain $\Omega_s$ which is fixed in
time. In the remaining ring domain $r \in (r_1,r_2)$ we model non-conducting
air in $\Omega_{air}$. When using
\[
  \psi(r) = \left \{
    \begin{array}{cl}
      1 & \mbox{for} \; r \in (r_0,r_1), \\[2mm] \displaystyle
      \frac{r_2-r}{r_2-r_1} & \mbox{for} \; r \in (r_1,r_2), \\[3mm]
      0 & \mbox{for} \; r \in (r_2,R),
    \end{array}
  \right.
\]
we can introduce the deformation
\begin{align}\label{Def Rotation}
  y(t) = \varphi(t,x) =
  \begin{pmatrix}
    y_1(t) \\[2mm]
    y_2(t)
  \end{pmatrix}
  = \, r \,
  \begin{pmatrix}
    \cos\Big(\phi + \alpha \, \psi(r) t \Big) \\[1mm]
    \sin\Big(\phi + \alpha \psi(r) t \Big)
\end{pmatrix} \in \Omega(t) \; \mbox{for} \; t \in (0,T)
\end{align}
which is a rotation of velocity $\alpha$ in the rotor, and which is
fixed in the stator. Here, $\Omega(t) := \varphi(t, \Omega)$ represents
the moving domain at time $t\in [0,T]$, and likewise we will use
$\Omega_r(t)$, $\Omega_{non}(t)$ and $\Omega_{con}(t)$.
In general, the velocity is given as
\begin{align*}
  v(y,t) = \frac{d}{dt} y(t) \,
  = \, \alpha \, \psi(r) \, r \,
  \begin{pmatrix}
    - \sin\left(\phi + \alpha \, \psi(r) \, t \right) \\[1mm]
    \cos\left(\phi + \alpha \, \psi(r) \, r \right)
  \end{pmatrix}
  = \, \alpha \, \psi(r) \, 
  \begin{pmatrix}
    - y_2(t) \\[2mm]
    y_1(t)
  \end{pmatrix}.
\end{align*}
A simple calculation shows, recall $r = \sqrt{y_1^2+y_2^2}$, that
\[
  \frac{\partial}{\partial y_1}v_1(y,t) = - \alpha \, y_2 \, \psi'(r) \,
  \frac{y_1}{\sqrt{y_1^2+y_2^2}}, \quad
  \frac{\partial}{\partial y_2}v_2(y,t) = \alpha \, y_1 \, \psi'(r) \,
  \frac{y_2}{\sqrt{y_1^2+y_2^2}},
\]
and hence, $\text{div}_{(y_1,y_2)}v(y_1,y_2,t)=0$ follows. With this
we can write Reynold's transport theorem as
\begin{eqnarray}
  && \hspace*{-5mm} \frac{d}{dt} \int_{\Omega(t)} u(y_1,y_2,t) \,
  dy_1 dy_2
  =   \int_{\Omega(t)} \Big( \nonumber
        \partial_t u(y_1,y_2,t) +
        \text{div}_{(y_1,y_2)} [u(y_1,y_2,t) \,
        v(y_1,y_2,t)] \Big) \,
        dy_1 dy_2 \\
  && \hspace*{-5mm} = \int_{\Omega(t)} \Big( \nonumber
        \partial_t u(y_1,y_2,t) +
        u(y_1,y_2,t) \, \text{div}_{(y_1,y_2)} 
        v(y_1,y_2,t)
        + v(y_1,y_2,t) \cdot
        \nabla_{(y_1,y_2)} u(y_1,y_2,t)
        \Big) \,
        dy_1 dy_2 \\
  && \hspace*{-5mm} = \int_{\Omega(t)} \Big( \nonumber
        \partial_t u(y_1,y_2,t) 
        + v(y_1,y_2,t) \cdot
        \nabla_{(y_1,y_2)} u(y_1,y_2,t)
        \Big) \,
        dy_1 dy_2 \\
  && \hspace*{-5mm} = \int_{\Omega(t)} 
        \frac{d}{dt} u(y_1,y_2,t) \,
        dy_1 dy_2 \label{Reynolds} \, .
\end{eqnarray}

\section{Space-time variational formulation} \label{sec:var_form}
We consider the evolution equation
\begin{align}\label{PDE ST}
  \sigma(y) \frac{d}{dt} u(y,t) - \mbox{div}_y [\nu(y) \nabla_y u(y,t)]
  = j_i(y,t) - \mbox{div}_y [\nu(y) M^\bot(y,t)] \quad \mbox{for}
  \; (y,t) \in Q,
\end{align}
where the space-time domain $Q$ is given by the deformation
\eqref{Def Rotation} as
\begin{align*}
  Q := \Big \{ (y,t) \in {\mathbb{R}}^3 : y = \varphi(t,x) \in \Omega(t),
  \, x \in \Omega \subset {\mathbb{R}}^2, \, t \in (0,T)\Big\},
\end{align*}
with homogeneous Dirichlet boundary conditions $u(x,t)=0$ for
$(y,t) \in \partial \Omega(t) \times (0,T)$, and with either the initial
condition $u(x,0)=0$ or the periodicity
condition $u(x,T)=u(x,0)$ for $x \in \Omega \backslash \Omega_{non}$.
Note that the partial differential equation \eqref{PDE ST} covers both
conducting ($\sigma>0$) and non-conducting materials ($\sigma=0$),
and the case of a fixed domain $(v=0)$ as well as the rotating
regions.

In order to introduce a variational formulation of \eqref{PDE ST} in the
space-domain $Q$ we first define the Bochner space
$Y:=L^2(0,T;H^1_0(\Omega(t))$ covering the homogeneous Dirichlet
boundary conditions, equipped with the norm
\begin{align*}
  \| z \|_Y^2 := \int_0^T \int_{\Omega(t)} \nu(y) \,
  |\nabla_y z(y,t)|^2 \, dy \, dt \, ,
\end{align*}
and the ansatz space
\begin{align*}
  X := \left \{ u \in Y : \sigma \frac{d}{dt}u \in Y^*, \;
  u(x,0) = 0 \text{ for } x \in \Omega_{con} \right\} \subset Y,
\end{align*}
in the case of homogeneous initial conditions, or
\begin{align*}
  X := \left \{ u \in Y : \sigma \frac{d}{dt}u \in Y^*, \;
  u(x,T) = u(x,0) \text{ for } x \in \Omega_{con} \right\} \subset Y
\end{align*}
in the case of a periodic behavior.
The graph norm in $X$ is given in both cases as
\begin{align*}
  \| u \|_X^2 := \| u \|_Y^2 + \| \sigma \frac{d}{dt} u \|_{Y^*}^2 =
  \| u \|^2_Y + \| w_u \|^2_Y,
\end{align*}
where $w_u \in Y$ is the unique solution of the variational formulation
\begin{align}\label{eqn:auxiliary_variational_form}
  \int_0^T \int_{\Omega(t)} \nu \, \nabla_y w_u \cdot \nabla_y z \, dy \, dt
  =
  \int_0^T \int_{\Omega(t)} \sigma \, \frac{d}{dt}u \, z \, dy \ dt \quad
  \mbox{for all} \; z \in Y.
\end{align}
Now, the space-time variational formulation of the evolution equation
\eqref{PDE ST} is to find $u \in X$ such that
\begin{align}\label{eqn:variational_form}
  b(u,z) := \int_0^T \int_{\Omega(t)} \left[ \sigma \, \frac{d}{dt}u \, z
  + \nu \, \nabla_y u \cdot \nabla_y z \right] \, dy \, dt =
  \int_0^T \int_{\Omega(t)} \left[ j_i \, z + M^\perp \cdot \nabla_y z \right]
  \, dy \, dt
\end{align}
is satisfied for all $z \in Y$. As in the case of a fixed domain $\Omega$,
see \cite{St15}, we conclude that the bilinear form $b(\cdot,\cdot)$
is bounded, satisfying
\begin{align*}
  |b(u,z)| \, \leq \, \sqrt{2} \, \|u\|_X \, \|z\|_Y \quad
  \mbox{for all} \; u \in X, z \in Y .
\end{align*}
Moreover, similar as in \cite{St15} and due to \eqref{Reynolds} we can
prove that the bilinear form $b(\cdot,\cdot)$ satisfies the inf-sup
stability condition
\begin{align}\label{inf-sup}
  \frac{1}{\sqrt{2}} \, \|u\|_X \leq
  \sup\limits_{0 \neq z \in Y} \frac{b(u,z)}{\|z\|_Y} \quad \mbox{for all} \;
  u \in X \, .
\end{align}
Indeed, for any given $u \in X$ let $w_u \in Y$ be the unique solution
of the variational formulation \eqref{eqn:auxiliary_variational_form}.
Due to $X \subset Y$ we can consider $z_u := u + w_u \in Y$ to obtain,
when using \eqref{eqn:auxiliary_variational_form} twice,
\begin{eqnarray*}
  b(u,z_u)
  & = & b(u,u+w_u) \, = \,
        \int_0^T \int_{\Omega(t)} \sigma \frac{d}{dt}u \, u \, dy \, dt +
        \int_0^T \int_{\Omega(t)} \nu \, \nabla_y u \cdot \nabla_y u \, dy \, dt
  \\
  & & \hspace*{25mm}
      + \int_0^T \int_{\Omega(t)} \sigma \frac{d}{dt}u \, w_u \, dy \, dt +
      \int_0^T \int_{\Omega(t)} \nu \, \nabla_y u \cdot \nabla_y w_u \, dy \, dt
  \\
  & & \hspace*{-2cm}
      = \, 2 \int_0^T \int_{\Omega(t)} \sigma \frac{d}{dt}u \, u \, dy \, dt
        +
        \int_0^T \int_{\Omega(t)} \nu \, \nabla_y u \cdot \nabla_y u \, dy \, dt
        +
        \int_0^T \int_{\Omega(t)} \nu \, \nabla_y w_u
        \cdot \nabla_y w_u \, dy \, dt \\
  && \hspace*{-2cm}
     = \int_0^T \int_{\Omega(t)} \sigma \, \frac{d}{dt} [u]^2 \, dy \, dt +
        \| u \|_Y^2 + \| w_u \|^2_Y \, \geq \, \| u \|_X^2 \, .
\end{eqnarray*}
Note that, since $\sigma$ is constant in time,
we can use \eqref{Reynolds} to conclude
\begin{eqnarray*}
  \int_0^T \int_{\Omega(t)} \sigma \, \frac{d}{dt} [u]^2 \, dy \, dt
  & = & \int_0^T \frac{d}{dt} \int_{\Omega(t)} \sigma \, [u]^2 \, dy \, dt
        \, = \, \left.
        \int_{\Omega(t)} \sigma(y) \, [u(y,t)]^2 \, dy \right|_0^T \\
  & = &  \int_{\Omega_{con}(T)} \sigma(y) \, [u(y,T)]^2 \, dy \, > \, 0
\end{eqnarray*}
in the case of initial conditions $u(x,0)=0$ for $x \in \Omega_{con}(0)$, or
\[
  \int_0^T \int_{\Omega(t)} \sigma \, \frac{d}{dt} [u]^2 \, dy \, dt \, = \,
  \int_{\Omega_{con}(T)} \sigma(y) \, [u(y,T)]^2 \, dy -
  \int_{\Omega_{con}(0)} \sigma(y) \, [u(y,0)]^2 \, dy \, = \, 0 
\]
in the case of periodicity $u(x,T)=u(x,0)$.

On the other hand, by the triangle and H\"older inequality we have
\[
  \| z_u \|_Y = \| u + w_u \|_Y \leq \| u \|_Y + \| w_u \|_Y
  \leq \sqrt{2} \, \sqrt{\| u \|_Y^2 + \| w_u \|_Y^2} =
  \sqrt{2} \, \| u \|_X ,
\]
i.e.,
\[
  b(u,z_u) \, \geq \, \| u \|_X^2 \, \geq \,
  \frac{1}{\sqrt{2}} \, \| u \|_X \| z_u \|_Y
\]
implies the inf-sup stability condition \eqref{inf-sup}.

More involved, and not as simple as in the static case, is to prove that the
bilinear form $b(\cdot,\cdot)$ is also surjective.

\begin{lemma}\label{Lemma surjective}
  For all $ z \in Y \backslash \{ 0 \}$ there exists a
  $\widetilde{u} \in X$ such that
  \[
    b(\widetilde{u},z) \neq 0 \, .
  \]
\end{lemma}

\begin{proof}
  We first consider the case of initial conditions $u(x,0)=0$ for
  $x \in \Omega_{con}$. Using the representation \eqref{Def Rotation}
  and for given $z \in Y \backslash \{ 0 \}$ we first define
  \[
    \widetilde{u}(y,t) =
    \widetilde{u}(\varphi(t,x),t) :=
    \int_0^t z(\varphi(s,x),s) \, ds,
    \quad \frac{d}{dt} \widetilde{u}(y,t) = z(y,t) \quad \mbox{for} \;
    y \in \Omega_{con}(t), t \in (0,T).
  \]
  By definition we have $\widetilde{u} \in X$ satisfying the initial
  condition $\widetilde{u}(x,0)=0$ for $x \in \Omega_{con}$ and
  \begin{eqnarray*}
    b(\widetilde{u},z)
    & = & \int_0^T \int_{\Omega(t)} \sigma(y)
        \frac{d}{dt} \widetilde{u}(y,t) \,
          z(y,t) \,
          dy \, dt
          + \int_0^T \int_{\Omega(t)} \nu(y) \, \nabla_y
        \widetilde{u}(y,t) \cdot \nabla_y
          z(y,t) \, dy \, dt \\
    & = &  
        \int_0^T \int_{\Omega_{con}(t)} \sigma(y) \,
          [z(y,t)]^2 \, dy \, dt +
        \int_0^T \int_{\Omega(t)}
        \nu(y) \, \nabla_y
          \widetilde{u}(y,t) \cdot
          \nabla_y \frac{d}{dt}
          \widetilde{u}(y,t) \,
        dy \, dt \, .
  \end{eqnarray*}
  Since the first term is obviously positive, it remains to treat the
  second term which involves an integral over
  $\Omega(t) = (\Omega_s \cap \Omega_{con}) \cup
  (\Omega_r(t) \cap \Omega_{con}(t)) \cup \Omega_{non}(t)$.
  In the stator domain $\Omega_s$, i.e., for $ r \in (r_2,R)$,
  we have $y(t) = x$ for all $t \in (0,T)$ and
  $v(y,t)=0$, and hence
  \begin{eqnarray*}
    \int_0^T \int_{\Omega_s \cap\Omega_{con}} \nu(x) \, \nabla_x
    \widetilde{u}_s(x,t) \cdot \nabla_x \partial_t
    \widetilde{u}_s(x,t) \, dx \, dt
    & = & \frac{1}{2} \int_0^T \frac{d}{dt} \int_{\Omega_s \cap\Omega_{con}}
          \nu(x) \, |\nabla_x \widetilde{u}_s(x,t)|^2 dx \,dt \\
    & = & \frac{1}{2} \int_{\Omega_s \cap\Omega_{con}} \nu(x) \,
        |\nabla_x \widetilde{u}_s(x,T)|^2 dx \geq 0 \, .
  \end{eqnarray*}
  Next we consider the rotor domain $\Omega_r(t)$, i.e., $r \in (r_0,r_1)$.
  For the total time derivative we then obtain
  \begin{eqnarray*}
    \frac{d}{dt} \widetilde{u}(y,t)
    & = & \frac{\partial}{\partial t} \widetilde{u}(y,t) +
          v(y,t)  \cdot
          \nabla_y \widetilde{u}(y,t) \\
    & = & \frac{\partial}{\partial t} \widetilde{u}(y,t) -
          \alpha \, y_2 \frac{\partial}{\partial y_1}
          \widetilde{u}(y,t) +
          \alpha \, y_1 \frac{\partial}{\partial y_2}
          \widetilde{u}(y,t).
  \end{eqnarray*}
  To compute the spatial gradient, we now consider
  \begin{eqnarray*}
    \frac{\partial}{\partial y_1} \frac{d}{dt}
    \widetilde{u}(y,t)
    & = & \frac{\partial}{\partial y_1} \left[
          \frac{\partial}{\partial t} \widetilde{u}(y,t) -
          \alpha \, y_2 \frac{\partial}{\partial y_1}
          \widetilde{u}(y,t) +
          \alpha \, y_1 \frac{\partial}{\partial y_2}
          \widetilde{u}(y,t) \right] \\
    & = & \frac{\partial}{\partial y_1} \frac{\partial}{\partial t}
          \widetilde{u}(y,t) -
          \alpha \, y_2 \frac{\partial^2}{\partial y_1^2}
          \widetilde{u}(y,t) +
          \alpha \, y_1 \frac{\partial}{\partial y_1}
          \frac{\partial}{\partial y_2} \widetilde{u}(y,t)
          + \alpha \, \frac{\partial}{\partial y_2}
          \widetilde{u}(y,t) \\
    & = & \frac{\partial}{\partial t} \frac{\partial}{\partial y_1} 
          \widetilde{u}(y,t) -
          \alpha \, y_2 \frac{\partial^2}{\partial y_1^2}
          \widetilde{u}(y,t) +
          \alpha \, y_1 \frac{\partial}{\partial y_2}
          \frac{\partial}{\partial y_1} \widetilde{u}(y,t)
          + \alpha \, \frac{\partial}{\partial y_2}
          \widetilde{u}(y,t) \\
    & = & \frac{d}{dt} \frac{\partial}{\partial y_1}
          \widetilde{u}(y,t)
          + \alpha \, \frac{\partial}{\partial y_2}
          \widetilde{u}(y,t),
  \end{eqnarray*}
  and
  \begin{eqnarray*}
    \frac{\partial}{\partial y_2} \frac{d}{dt}
    \widetilde{u}(y,t)
    & = & \frac{\partial}{\partial y_2} \left[
          \frac{\partial}{\partial t} \widetilde{u}(y,t) -
          \alpha \, y_2 \frac{\partial}{\partial y_1}
          \widetilde{u}(y,t) +
          \alpha \, y_1 \frac{\partial}{\partial y_2}
          \widetilde{u}(y,t) \right] \\
    & = & \frac{\partial}{\partial y_2} \frac{\partial}{\partial t}
          \widetilde{u}(y,t) -
          \alpha \, y_2 \frac{\partial}{\partial y_2}
          \frac{\partial}{\partial y} \widetilde{u}(y,t)
          + \alpha y_1 \frac{\partial^2}{\partial y_2^2}
          \widetilde{u}(y,t)
          - \alpha \, \frac{\partial}{\partial y_1}
          \widetilde{u}(y,t) \\
    & = & \frac{\partial}{\partial t} \frac{\partial}{\partial y_2} 
          \widetilde{u}(y,t) -
          \alpha \, y_2 \frac{\partial}{\partial y_1}
          \frac{\partial}{\partial y_2}\widetilde{u}(y,t)
          + \alpha \, y_1 \frac{\partial^2}{\partial y_2^2}
          \widetilde{u}(y,t)
          - \alpha \, \frac{\partial}{\partial y_1}
          \widetilde{u}(y,t) \\
    & = & \frac{d}{dt} \frac{\partial}{\partial y_2}
          \widetilde{u}(y,t)
          - \alpha \, \frac{\partial}{\partial y_1}
          \widetilde{u}(y,t) .
  \end{eqnarray*}
  Hence we obtain
  \begin{eqnarray*}
    \nabla_y \widetilde{u}(y,t) \cdot
    \nabla_y \frac{d}{dt} \widetilde{u}(y,t)
    & = & \frac{\partial}{\partial y_1} \widetilde{u}(y,t)
          \frac{\partial}{\partial y_1} \frac{d}{dt}
          \widetilde{u}(y,t) +
          \frac{\partial}{\partial y_2} \widetilde{u}(y,t)
          \frac{\partial}{\partial y_2} \frac{d}{dt}
          \widetilde{u}(y,t) \\
    && \hspace*{-4.5cm} =
       \frac{\partial}{\partial y_1} \widetilde{u}(y,t)
       \left[ \frac{d}{dt} \frac{\partial}{\partial y_1}
       \widetilde{u}(y,t)
       + \alpha \, \frac{\partial}{\partial y_2}
       \widetilde{u}(y,t) \right]
       + \frac{\partial}{\partial y_2} \widetilde{u}(y,t)
       \left[ \frac{d}{dt} \frac{\partial}{\partial y_2}
       \widetilde{u}(y,t)
       - \alpha \, \frac{\partial}{\partial y_1}
       \widetilde{u}(y,t) \right] \\
    && \hspace*{-4.5cm} =
       \frac{\partial}{\partial y_1} \widetilde{u}(y,t) 
       \frac{d}{dt} \frac{\partial}{\partial y_1}
       \widetilde{u}(y,t)
       + \frac{\partial}{\partial y_2} \widetilde{u}(y,t) 
       \frac{d}{dt} \frac{\partial}{\partial y_2}
       \widetilde{u}(y,t) \\
    && \hspace*{-4.5cm} =
       \nabla_y \widetilde{u}(y,t) \cdot \frac{d}{dt}
       \nabla_y \widetilde{u}(y,t),
  \end{eqnarray*}
  and therefore,
  \begin{eqnarray*}
    && \int_0^T \int_{\Omega_r(t) \cap\Omega_{con}(t)} \nu(y) \, \nabla_y
       \widetilde{u}(y,t) \cdot \nabla_y \frac{d}{dt} 
       \widetilde{u}(y,t) \, dy \, dt \\
    && \hspace*{3cm}
       = \,  \int_0^T \int_{\Omega_r(t) \cap\Omega_{con}(t)} \nu(y) \,
       \nabla_y \widetilde{u}(y,t) \cdot \frac{d}{dt}\nabla_y 
       \widetilde{u}(y,t) \, dy \, dt \\
    && \hspace*{3cm}
       = \, \frac{1}{2} \int_0^T \int_{\Omega_r(t) \cap\Omega_{con}(t)}
       \nu(y) \, \frac{d}{dt} \,
       |\nabla_y \widetilde{u}(y,t)|^2
       \, dy \, dt  \\
    && \hspace*{3cm} = \,
       \frac{1}{2} \int_0^T \frac{d}{dt} \int_{\Omega_r(t) \cap\Omega_{con}(t)}
       \nu(y) \, |\nabla_y
       \widetilde{u}(y,t)|^2
       \, dy \, dt \\
    && \hspace*{3cm} = \, \frac{1}{2} \int_{\Omega_r(T) \cap\Omega_{con}(T)}
       \nu(y) \, |\nabla_y
       \widetilde{u}(y,T)|^2 \, dy
       \geq 0 
  \end{eqnarray*}
  follows, i.e.,
  \[
    b(\widetilde{u},z) \geq
    \int_0^T \int_{\Omega_{con}(t)} \sigma(y) \, [z(y,t)]^2 \, dy \, dz +
    \int_0^T \int_{\Omega_{non}(t)} \nu(y) \, \nabla_y \widetilde{u}(y,t)
    \cdot \nabla_y z(y,t) \, dy \, dt \, .
  \]
  It remains to define $\widetilde{u}$ in the non-conduction regions in
  suitable way. In any non-conducting subregion we can write
  \begin{eqnarray*}
    && \int_0^T \int_{\Omega_{non}(t)}
    \nu(y) \, \nabla_y \widetilde{u}(y,t) \cdot \nabla_y
    z(y,t) \, dy \, dt
    \\
    && \hspace*{1cm} =
    \int_0^T \int_{\Omega_{non}(t)} [z(y,t)]^2 \, d y \, dt
    +
    \int_0^T \int_{\partial \Omega_{non}(t)} n_y \cdot [
    \nu(y) \,
    \nabla_y \widetilde{u}(y,t) ] \, z(y,t) \, ds_y dt,
  \end{eqnarray*}
  when $\widetilde{u}$ is a solution of the quasi-static partial
  differential equation
  \[
    - \mbox{div}_y [\nu(y) \nabla_y \widetilde{u}(y,t)] =
    z(y,t) \quad \mbox{for} \; y \in \Omega_{non}(t), \, t \in (0,T).
  \]
  To ensure $\widetilde{u} \in L^2(0,T;H^1_0(\Omega(t)))$ we formulate
  the boundary conditions
  $\widetilde{u}_{|\Omega_{non}(t)} = \widetilde{u}_{|\Omega_{con}(t)}$ 
  on $\partial \Omega_{non}(t) \cap \partial \Omega_{con}(t)$ and
  $\widetilde{u}_{|\Omega_{non}(t)} = 0$ 
  on $\partial \Omega_{non}(t) \cap \partial \Omega(t)$
  for all $t \in (0,T)$. The solution of this
  quasi-static Dirichlet boundary value problem implies the
  Dirichlet to Neumann map
  \[
    n_y \cdot \Big[ \nu(y) \, \nabla_y \widetilde{u}(y,t) \Big] =
    (S \widetilde{u})(y,t) \quad
    \mbox{for} \; y \in \partial \Omega_{non}(t), \, t \in (0,T),
  \]
  with the Steklov--Poincar\'e operator
  $S : H^{1/2}(\partial \Omega_{non}(t)) \to H^{-1/2}(\partial \Omega_{non}(t))$.
  Since the shape of $\Omega_{non}(t)$ is fixed, $S$ does not
  depend on time. On the other hand, since $S$ is self-adjoint and
  positive semi-definite, we can factorize $S$ to write
  \begin{eqnarray*}
    \int_0^T \int_{\partial \Omega_{non}(t)}
    (S \widetilde{u})(y,t) \, z(y,t) \, ds_y dt
    & = &
    \int_0^T \int_{\partial \Omega_{non}(t)}
    (S^{1/2} \widetilde{u})(y,t) \,
          (S^{1/2} z)(y,t) \, ds_y dt \\
    & = & 
    \int_0^T \int_{\partial \Omega_{non}(t)}
    (S^{1/2} \widetilde{u})(y,t) \, \frac{d}{dt}
    (S^{1/2} \widetilde{u})(y,t) \, ds_y dt \\
    & = & \frac{1}{2}
    \int_0^T \frac{d}{dt} \int_{\partial \Omega_{non}(t)}
        \Big[
        (S^{1/2} \widetilde{u})(y,t) \Big]^2
        \, ds_y dt \\
    & = & \frac{1}{2}
        \int_{\partial \Omega_{non}(t)}
        \Big[
        (S^{1/2} \widetilde{u})(y,T) \Big]^2
        \, ds_y dt \, \geq \, 0 \, .
  \end{eqnarray*}
  This finally gives
  \[
    b(\widetilde{u},z)
    \geq
    \int_0^T \int_{\Omega_{con}(t)} \sigma(y) \,
    [z(y,t)]^2 \, dy \, dt +
    \int_0^T \int_{\Omega_{non}(t)} [z(y,t)]^2 \,
    d s_y \, dt > 0 .
  \]
  It remains to consider the case of the periodicity condition
  $\widetilde{u}(x,T)=\widetilde{u}(x,0)$ for $x \in \Omega_{con}$.
  In order to construct a suitable $\widetilde{u}$ in this case,
  let us recall that in the case of the initial condition
  $\widetilde{u}(x,0)$ for $x \in \Omega_{con}$ we have constructed
  $\widetilde{u}$ as solution of the ordinary differential equation
  \[
    \frac{d}{dt}
    \widetilde{u}(\varphi(t,x),t)
    =
    z(\varphi(t,x),t) \quad
    \mbox{for} \; t \in (0,T), \quad
    \widetilde{u}(\varphi(0,x),0)
    = 0 .
  \]
  To allow for a periodic behavior of the solution $\widetilde{u}$,
  we now consider the ordinary differential equation
  \[
    \frac{d}{dt}
    \widetilde{u}(\varphi(t,x),t)
    +
    \widetilde{u}(\varphi(t,x),t)
    =
    z(\varphi(t,x),t) \quad
    \mbox{for} \; t \in (0,T), 
  \]
  with the solution
  \[
    \widetilde{u}(\varphi(t,x),t)
    = e^{-t} \,
    \widetilde{u}(\varphi(0,x),0)
    +
    \int_0^t e^{s-t} z(\varphi(s,x),s) \, ds .
  \]
  From the periodicity condition $\widetilde{u}(x,T) = \widetilde{u}(x,0)$
  we then conclude
  \[
    \widetilde{u}(\varphi(0,x),0)
    =
    \widetilde{u}(\varphi(T,x),T)
    = e^{-T} \, \widetilde{u}(\varphi(0,x),0)
    +
    \int_0^T e^{s-T} z(\varphi (s,x),s) \, ds ,
  \]
  i.e.,
  \[
    \widetilde{u}(\varphi(0,x),0)
    := \frac{1}{1-e^{-T}}
    \int_0^T e^{s-T} z(\varphi
    (s,x),s) \, ds .
  \]
  By construction we have $\widetilde{u} \in X$, and hence
  \begin{eqnarray*}
    && \int_0^T \int_{\Omega_{con}(t)}
    \sigma(y) \, \frac{d}{dt} \widetilde{u}(y,t) \, z(y,t) \, dy \, dt +
    \int_0^T \int_{\Omega_{con}(t)} \nu(y) \,
    \nabla_y \widetilde{u}(y,t) \cdot \nabla_y z(y,t) \,
       dy \, dt \\
    && \hspace*{0.5cm}
       = \int_0^T \int_{\Omega_{con}(t)}
       \sigma(y) \, \frac{d}{dt} \widetilde{u}(y,t) \,
       \left[ \frac{d}{dt} \widetilde{u}(y,t) + \widetilde{u}(y,t)
       \right] \, dy \, dt \\
    && \hspace*{1cm} +
       \int_0^T \int_{\Omega_{con}(t)} \nu(y) \,
       \nabla_y \widetilde{u}(y,t) \cdot \nabla_y \left[
       \frac{d}{dt} \widetilde{u}(y,t) + \widetilde{u}(y,t)
       \right] \, dy \, dt \\
    && \hspace*{0.5cm}
       = \int_0^T \int_{\Omega_{con}(t)}
       \sigma(y) \,
       \left[ \frac{d}{dt} \widetilde{u}(y,t) 
       \right]^2 \, dy \, dt
       +
       \int_0^T \int_{\Omega_{con}(t)} \nu(y) \,
       |\nabla_y \widetilde{u}(y,t)|^2 \, dy \, dt \\
    && \hspace*{-5mm} +
       \int_0^T \int_{\Omega_{con}(t)}
       \sigma(y) \, \frac{d}{dt} \widetilde{u}(y,t) \,
       \widetilde{u}(y,t) \, dy \, dt +
       \int_0^T \int_{\Omega_{con}(t)} \nu(y) \,
       \nabla_y \widetilde{u}(y,t) \cdot \nabla_y 
       \frac{d}{dt} \widetilde{u}(y,t) \, dy \, dt .
  \end{eqnarray*}
Now the assertion follows as in the previous case.
\end{proof}

To summarize, all assumptions of the Babu\v ska--Ne\v cas theorem
\cite{BaAz72,ErnGuermond:2004,Ne62} are satisfied, which finally ensures unique
solvability of the space-time variational formulation
\eqref{eqn:variational_form}.

\section{Space-time finite element discretization}
\label{sec:space_time_fe_dis}
For the space-time finite element discretization of the variational
formulation \eqref{eqn:variational_form} we introduce conforming
finite dimensional spaces $X_h \subset X$ and $Y_h \subset Y$ where we
assume as in the continuous case $X_h \subset Y_h$. For our specific
purpose we even consider
\begin{align*}
  X_h = Y_h := S_h^1(Q_h) \cap X = \text{span}\{\varphi_k\}_{k=1}^M
\end{align*}
as the space of piecewise linear and continuous basis functions $\varphi_k$
which are defined with respect to some admissible locally quasi-uniform
decomposition $Q_h = \{\tau_\ell\}_{\ell=1}^N$ of the space-time domain $Q$
into shape-regular simplicial finite elements $\tau_\ell$ of mesh size
$h_\ell$, see e.g. \cite{NeKa19,St15}, and Fig. \ref{fig:motor_body}
for a space-time finite element mesh of a rotating electric motor.

\begin{figure}[tbhp]
\begin{center}
\begin{minipage}{0.75\linewidth}
\includegraphics[width=1\linewidth]{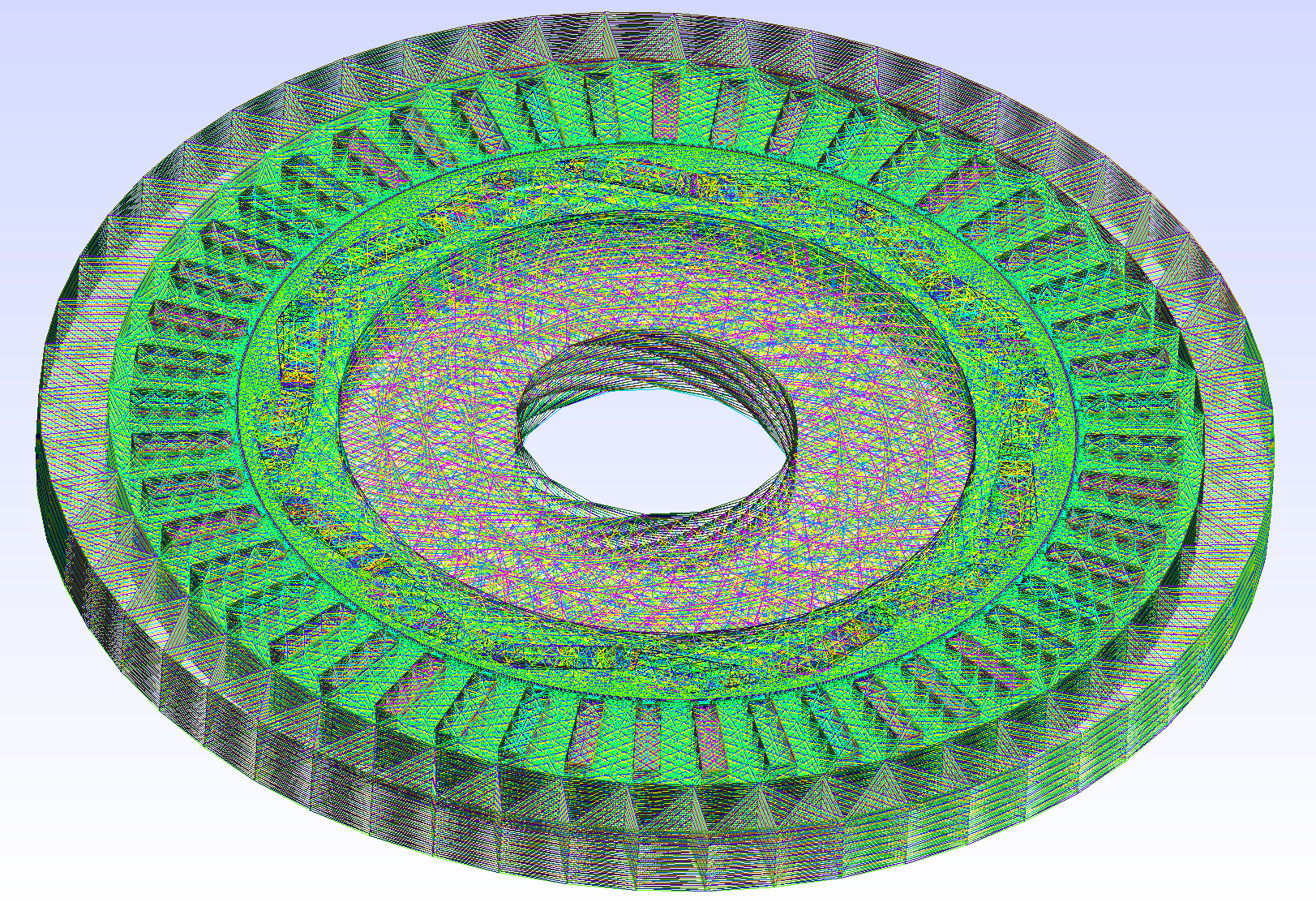} 
\end{minipage}
\end{center} 
\caption{The space-time mesh of an 90 degrees rotating electric motor
  with 16 magnets and 48 coils generated by Gmsh \cite{gmsh}. The mesh
  is divided into 30 time slices in temporal direction and consists
  of 333 288 nodes and 1 978 689 elements.}
\label{fig:motor_body}
\end{figure}

The Galerkin space-time finite element variational formulation of
\eqref{eqn:variational_form} reads to find $u_h \in X_h$, such that
\begin{align}\label{eqn:discrete_variational_form}
  \int_0^T \int_{\Omega(t)} \left[ \sigma \, \frac{d}{dt}u_h \, z_h
  + \nu \, \nabla_y u_h \cdot \nabla_y z_h \right] \, dy \, dt =
  \int_0^T \int_{\Omega(t)} \left[ j_i \, z_h + M^\perp \cdot \nabla_y z_h \right]
  \, dy \, dt
\end{align}
is satisfied for all $z_h \in Y_h$. To ensure unique solvability of
\eqref{eqn:discrete_variational_form} and to derive related error
estimates we proceed as in the case of a fixed domain \cite{St15}.
For any $u \in X$ we define $w_{uh} \in Y_h$ as the unique solution of the
Galerkin variational formulation
\begin{align}\label{eqn:discrete_auxiliary_variational_form}
  \int_0^T \int_{\Omega(t)} \nu \, \nabla_y w_{uh} \cdot \nabla_y z_h \, dy \, dt
  = \int_0^T \int_{\Omega(t)} \sigma \, \frac{d}{dt} u \, z_h \, dy \, dt
  \quad \mbox{for all} \; z_h \in Y_h
\end{align}
in order to define the discrete norm
\begin{align}\label{eqn:discrete_norm_of_X}
  \| u \|^2_{X_h} := \| u \|^2_Y + \| w_{uh}\|^2_Y \leq
  \| u \|^2_Y + \| w_u \|^2_Y = \|u \|^2_X \quad \mbox{for all} \; u \in X.
\end{align}
Correspondingly, for $u_h \in X_h$ we define $ w_{u_hh} \in Y$, and hence
we can consider, due to $X_h \subset Y_h$, the particular test function
$ z_h := u_h + w_{u_hh} \in Y_h$ to conclude, as in the continuous case,
\[
  b(u_h,u_h+w_{u_hh}) \geq
  \| u_h \|_Y^2 + \| w_{u_hh} \|^2_Y = \| u_h \|^2_{X_h},
  \quad
  \| z_h \|_Y \leq \| u_h \|_y + \| w_{u_hh} \|_Y \leq \sqrt{2} \,
  \| u_h \|_{X_h},
\]
and therefore the discrete inf-sup condition
\begin{align}\label{eqn:discrete_inf_sup_condition}
  \frac{1}{\sqrt{2}} \, \| u_h \|_{X_h} \leq \sup
  \limits_{0 \neq z_h \in Y_h} \frac{b(u_h,z_h)}{\|z_h\|_Y} \quad
  \mbox{for all} \; u_h \in X_h
\end{align}
follows. From \eqref{eqn:discrete_inf_sup_condition} we obtain unique
solvability of the Galerkin variational formulation 
\eqref{eqn:discrete_variational_form}, and due to
\[
  \frac{1}{\sqrt{2}} \, \| u_h \|_{X_h} \leq \sup
  \limits_{0 \neq z_h \in Y_h} \frac{b(u_h,z_h)}{\|z_h\|_Y} =
  \sup \limits_{0 \neq z_h \in Y_h} \frac{b(u,z_h)}{\|z_h\|_Y} 
  \leq \sqrt{2} \, \| u \|_X
\]
we conclude the boundedness of the Galerkin projection
$u_h = G_h u$, i.e.,
\[
  \| G_h u \|_{X_h} = \| u_h \|_{X_h} \leq 2 \, \| u \|_X \quad
  \mbox{for all} \; u \in X .
\]
From this we further obtain
\[
  \| u - u_h \|_{X_h} \leq \| u - z_h \|_{X_h} + \| G_h(z_h-u) \|_{X_h} \leq
  3 \, \| u - z_h \|_X \quad \mbox{for all} \; z_h \in X_h,
\]
i.e., we have Cea's lemma
\begin{align}\label{eqn:Ceas_lemma}
  \| u - u_h\|_{X_h} \leq 3 \inf \limits_{z_h \in X_h} \|u - z_h\|_X .
\end{align}
When using standard finite element error estimates
\cite{BrennerScott:2008,St08}
for piecewise linear approximations we finally conclude the error estimate 
\[
\| u - u_h \|_Y \, \leq \, c \, h \, |u|_{H^2(Q)}
\]
when assuming $u \in H^2(Q)$, see \cite{St15} for the case of a fixed
domain.

The Galerkin space-time finite element variational formulation
\eqref{eqn:discrete_variational_form} results in a huge linear system
of algebraic equations, which has to be solved efficiently, and in parallel.

	\section{Numerical experiments}\label{sec:numerical_experiments}
In this section we provide some numerical results in order to illustrate
the applicability, the accuracy and the efficiency of the proposed approach.

We consider the electric motor as shown in Fig.~\ref{Bild Motor}, where both
the rotor and the stator are made of laminated iron sheets, with $16$ magnets within the rotor
and $48$ coils within the stator. Between the rotor and the stator there is a
thin air gap, and also at the ends of the magnets air pockets are included.
The material parameters for the electric conductivity $\sigma$ and the magnetic reluctivity $\nu$ for the different materials are given in Table~\ref{table:material_paramter}. Note that the electric conductivity for iron and for the coils as given in Table \ref{table:material_paramter} are
chosen to be zero, since the materials in the electric motor are laminated.
Moreover, we account for saturation of the ferromagnetic material, thus the reluctivity $\nu_{\text{iron}}$ in iron is a nonlinear function of the magnitude of the magnetic flux density $|B| = |\nabla u|$. Hence, the variational problem
\eqref{eqn:discrete_variational_form} is a nonlinear problem.
The nonlinear reluctivity is computed from a spline interpolation of
given discrete values for the $BH$-curve representing the relationship between magnetic flux density $B$ and magnet field strength $H$ in a ferromagnetic material. It follows from physical properties of $BH$-curves that the corresponding reluctivity function $\nu_{\text{iron}}$ satisfies a Lipschitz and strong monotonicity condition, see \cite{PechsteinJuettler2006}.
In the coils we are given a three-phase alternating current with an amplitude of 1555A, from which the impressed current density $j_i$ is obtained after dividing by the coil area. The magnetization $M$ in the permanent magnets is constant for each of the magnets and for each pair of magnets (see Fig. \ref{Bild Motor}) points alternatingly inwards or outwards. The magnetization $M$ is then given by the unit vector representing a magnet's magnetization direction multiplied with the remanence flux density $B_R = 1.216$.

The motor is pulled up in time, where the rotation of the rotating parts,
i.e., the rotor, the magnets and the air around the magnets, is already
considered within the mesh for a $90$ degree rotation. As before, the time
component is treated as the third spatial dimension with a time span
$(0,T)$, $T=0.015$ seconds. Note that this corresponds to a rotational
speed of $1000$ rotations per minute. Moreover, $30$ time slices are
inserted in order to have a good temporal resolution, where the mesh is
completely unstructured within the time slices, see Fig.~\ref{fig:motor_body}.
The space-time finite element mesh consists of $978 689$ tetrahedral
finite elements, and $333 288$ nodes.

\begin{table}
  \caption{Material parameter to desribe the electric motor.}
  \label{table:material_paramter}
\begin{center}
\begin{tabular}[tbhp]{ |ccc| } 
 \hline
 material & $\sigma$ & $\nu$ \\
 \hline
 air & 0 & $10^7 / (4 \pi) $  \\ 
 coils & 0 & $10^7 / (4 \pi)$ \\
 magnets & $10^6$ & $10^7 / (4.2 \pi)$ \\ 
 iron & 0 & $\nu_{\text{iron}}(|\nabla u|)$ \\  
 \hline
\end{tabular}
\end{center}
\end{table}

We solve the resulting system in parallel, using a mesh decomposition
method provided by the finite element library Netgen/NGSolve \cite{netgen}.
For our purpose, MPI parallelization is used, however the computations are
done on one computer with 384 GiB RAM and two Intel Xenon Gold 5218 CPU’s with 20 cores each.
For the solution of the nonlinear problem we use Newton's method with damping where the linearized system of every Newton step is solved with
GMRES supported by PETSc \cite{DaPaKlCo2011}.

In a first numerical experiment, we solve a linear approximation to the nonlinear problem at hand by replacing the nonlinear magnetic reluctivity function $\nu_{\text{iron}}(|\nabla u|)$ by a constant $\nu_1 = 10^7/(5100 \cdot 4 \pi)$, which is a realistic approximation when saturation of the material does not occur. The solution to this linear problem including homogeneous initial conditions is displayed in Figure~\ref{fig:sol_lin_prob_initial_val}. Here, we made cross sections in
temporal directions at specific time points. In Table~\ref{table:speedup_table_linprob_mumps} and
Table~\ref{table:speedup_table_linprob_gmres} the computational times with
respect to the number of cores are given for the linear problem with
homogeneous initial conditions. Next we used the solution of this linearized problem as initial guess in Newton's method for solving the nonlinear problem with the reluctivity function $\nu_{\text{iron}}=\nu_{\text{iron}}(|\nabla u|)$. The solution of the initial value problem for different points in time is depicted in Figure \ref{fig:sol_motor_zero_initial_conditions}.
Table~\ref{table:speedup_table_nonlinprob_gmres}
shows the computational time with respect to the number of cores of the
Newton method stopped after 100 iterations. The initial value for the
Newton method is the solution of the linear problem with zero initial
conditions, as visualized in Fig.~\ref{fig:sol_lin_prob_initial_val}.
Moreover, the solutions for the nonlinear problem with periodic temporal conditions
are displayed in Fig.~\ref{fig:sol_motor_periodic_conditions}, but this problem
was not solved in parallel.

\begin{figure}[tbhp]
\begin{minipage}{0.47\linewidth}
\begin{center}
\includegraphics[width=1\linewidth]{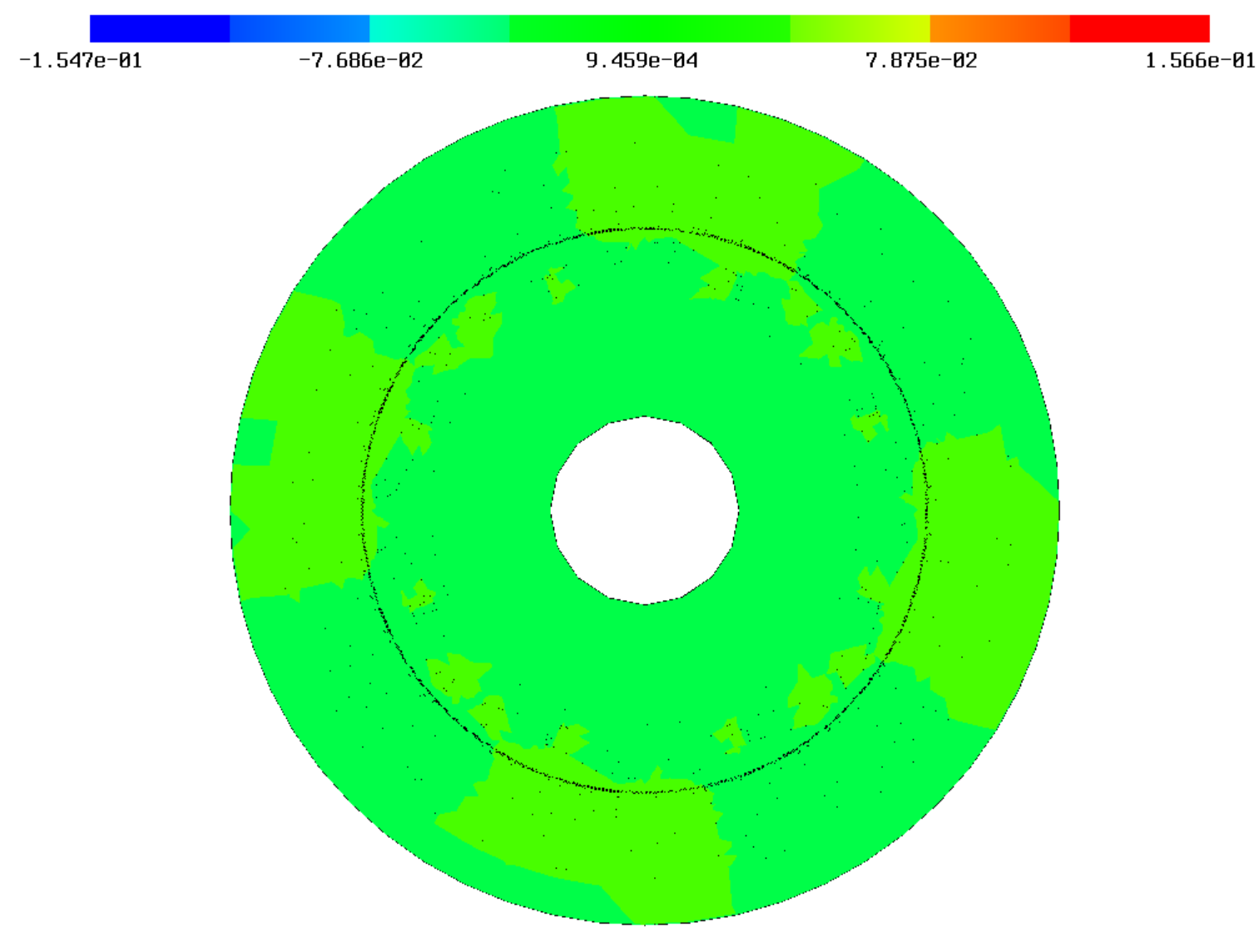} 
\textit{Solution at time $t = 0.0$.}
\end{center} 
\end{minipage}
\hfill
\vspace{0.2 cm}
\begin{minipage}{0.47\linewidth}
\begin{center}
\includegraphics[width=1\linewidth]{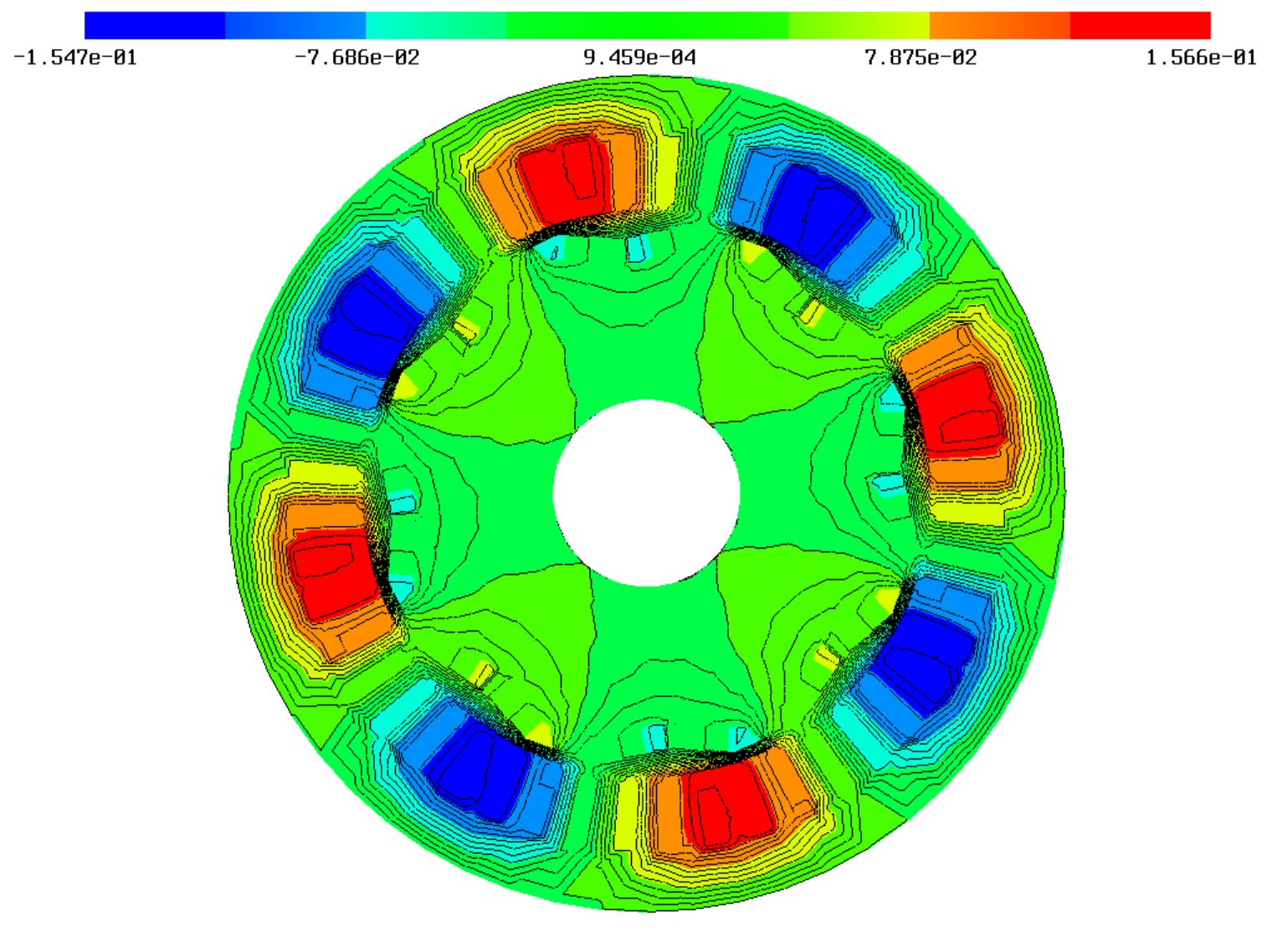} 
\textit{Solution at time $t = 0.0045$.}
\end{center}
\end{minipage}
\vfill
\vspace{0.2 cm}
\begin{minipage}{0.47\linewidth}
\begin{center}
\includegraphics[width=1\linewidth]{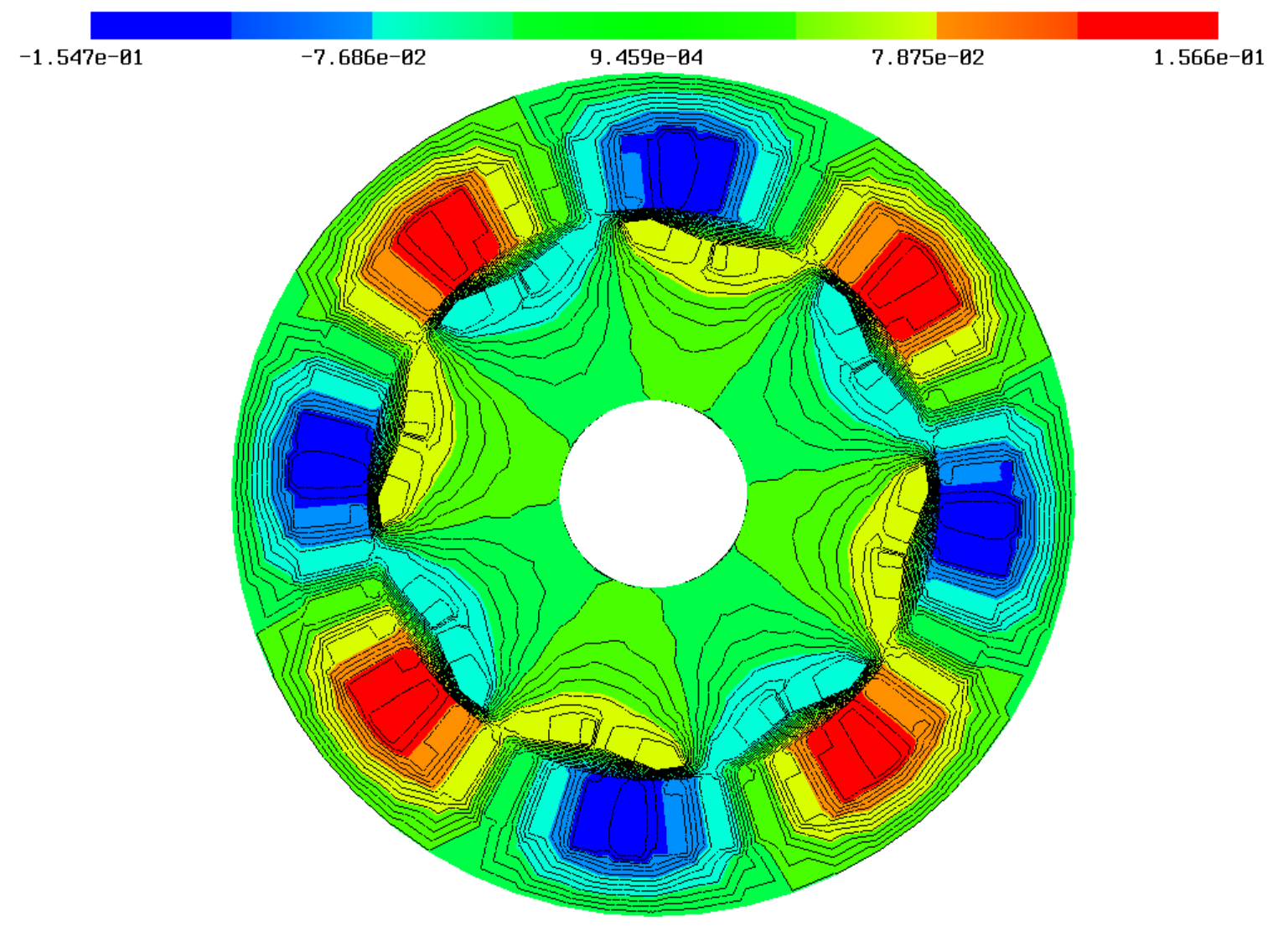} 
\textit{Solution at time $t = 0.009$.}
\end{center}
\end{minipage}
\hfill
\begin{minipage}{0.47\linewidth}
\begin{center}
\includegraphics[width=1\linewidth]{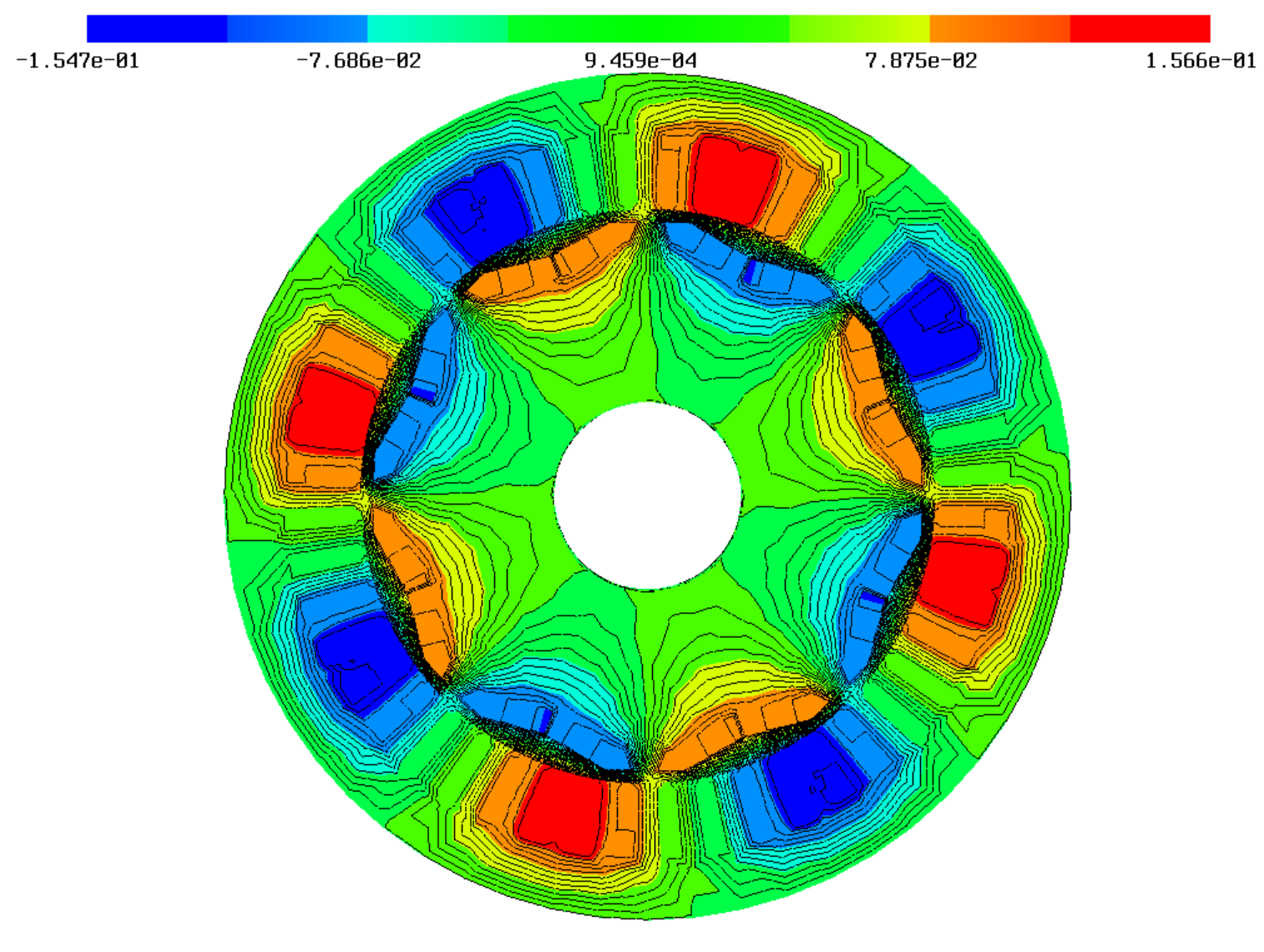} 
\textit{Solution at time $t = 0.015$.}
\end{center}
\end{minipage}
\caption{Cross sections of the solution for specific time points of the linear problem with zero initial conditions, which is considered as the start value for Newton's method.}
\label{fig:sol_lin_prob_initial_val}
\end{figure}

\begin{table}
\caption{Computational times in seconds of the linear problem with homogeneous initial conditions solved with MUMPS provided by PETSc \cite{DaPaKlCo2011}.}
\label{table:speedup_table_linprob_mumps}
\begin{center}
\begin{tabular}[tbhp]{ |c|c|c|c|c|c| } 
 \hline
 number of cores & 1 & 2 & 4 & 8 & 16\\
 \hline
 time in seconds & 14.12 & 12.2 & 10.75 & 9.63 & 10.17 \\ 
 \hline
\end{tabular}
\end{center}
\end{table}

\begin{table}
\caption{Computational times in seconds of the linear problem with homogeneous initial conditions solved with GMRES provided by PETSc \cite{DaPaKlCo2011} up to 1000 iterations.}
\label{table:speedup_table_linprob_gmres}
\begin{center}
\begin{tabular}[tbhp]{ |c|c|c|c|c|c| } 
 \hline
 number of cores & 1 & 2 & 4 & 8 & 16\\
 \hline
 time in seconds & 16.5 & 11.86 & 7.87 & 4.79 & 3.23 \\ 
 \hline
\end{tabular}
\end{center}
\end{table}

\begin{table}
\caption{Computational times in seconds of the nonlinear problem with homogeneous initial conditions solved with GMRES with 250 iterations in every Newton iteration with a maximum of 100 Newton iterations.}
\label{table:speedup_table_nonlinprob_gmres}
\begin{center}
\begin{tabular}[tbhp]{ |c|c|c|c|c|c| }
 \hline
 number of cores & 1 & 2 & 4 & 8 & 16\\
 \hline
 time in seconds & 9952 & 5103 & 2761 & 1463 & 848 \\
 \hline
\end{tabular}
\end{center}
\end{table}

\begin{figure}[tbhp]
\begin{minipage}{0.47\linewidth}
\begin{center}
\includegraphics[width=1\linewidth]{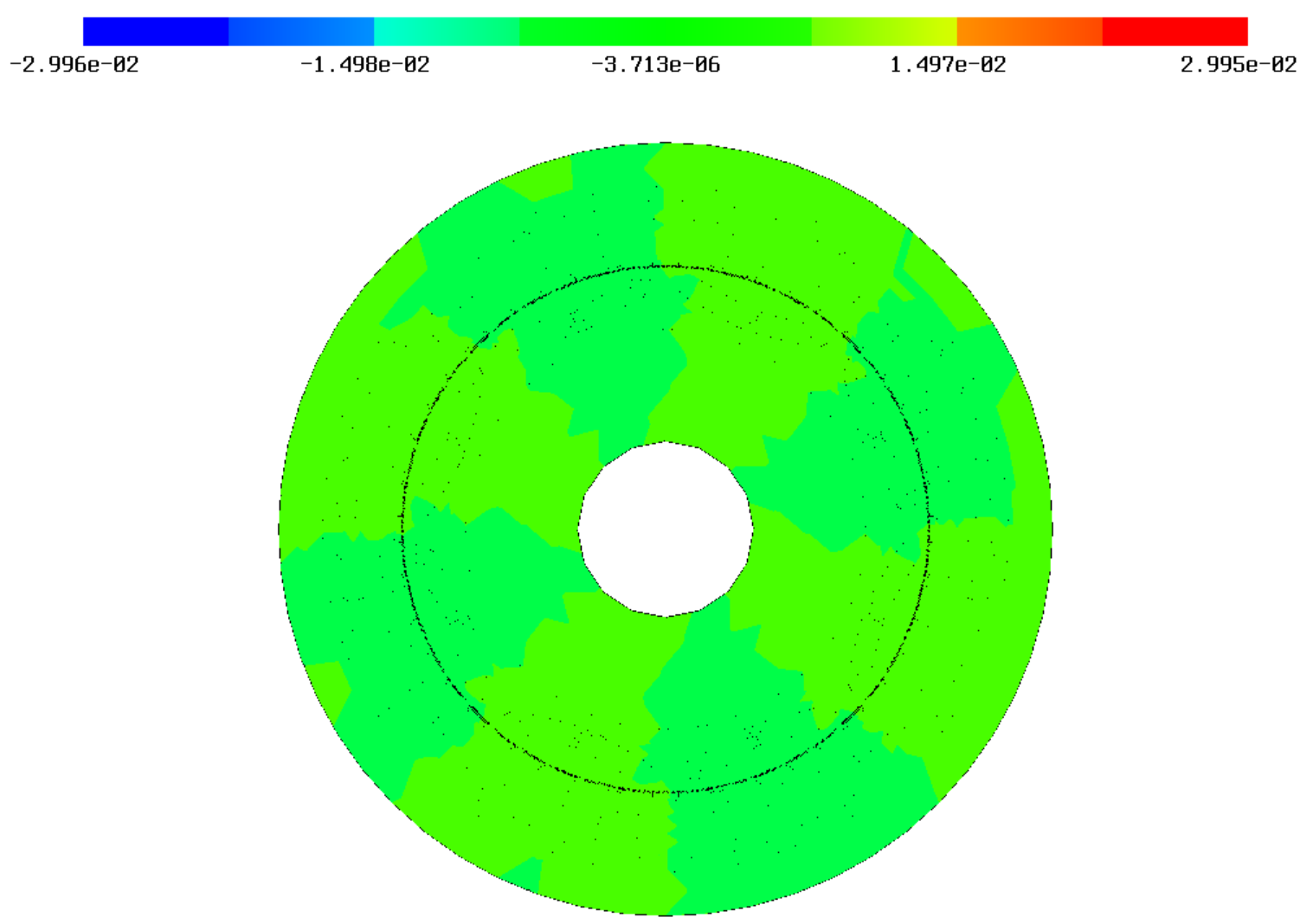} 
\textit{Solution at time $t = 0.0$.}
\end{center} 
\end{minipage}
\hfill
\vspace{0.2 cm}
\begin{minipage}{0.47\linewidth}
\begin{center}
\includegraphics[width=1\linewidth]{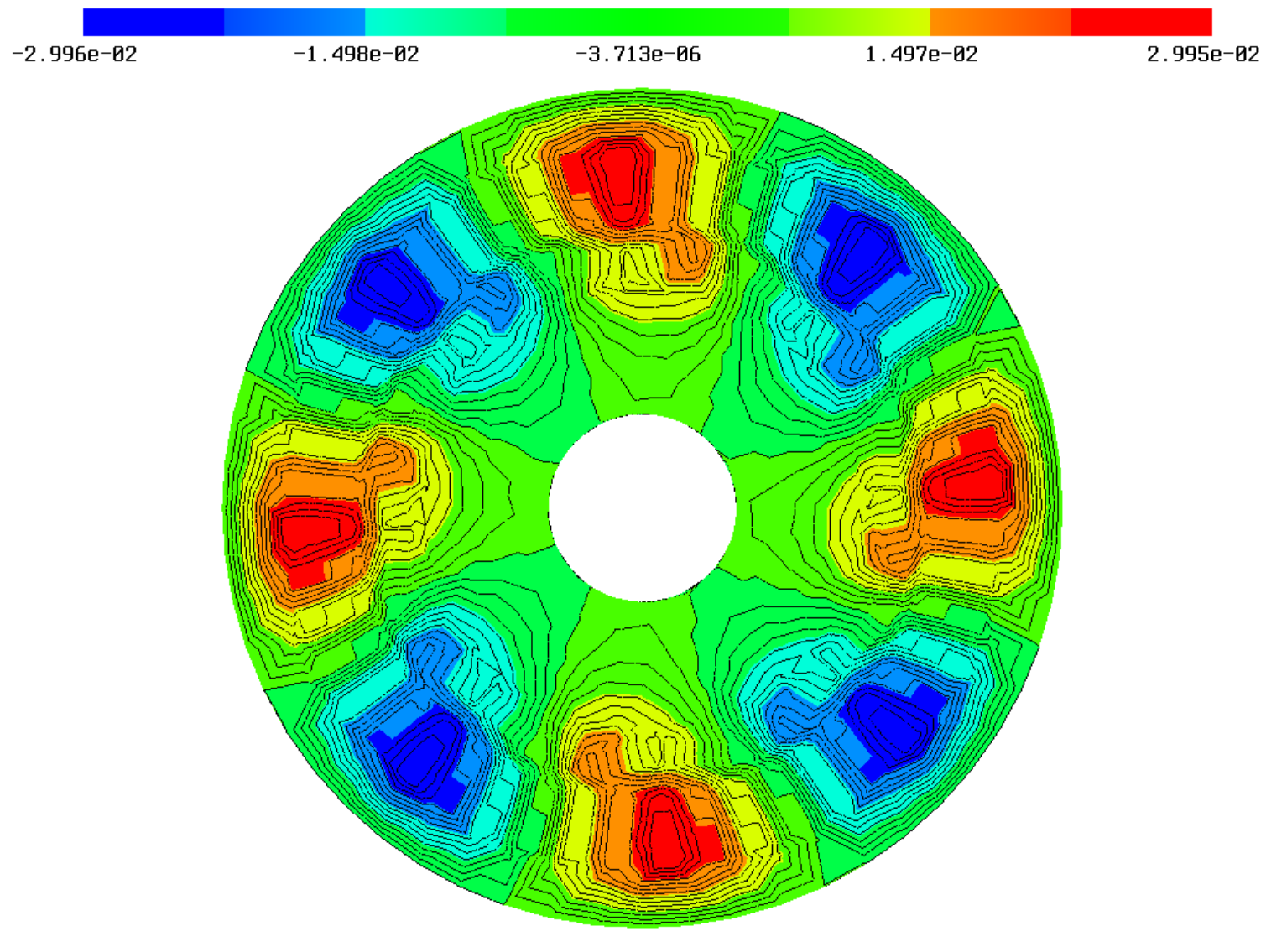} 
\textit{Solution at time $t = 0.0045$.}
\end{center}
\end{minipage}
\vfill
\vspace{0.2 cm}
\begin{minipage}{0.47\linewidth}
\begin{center}
\includegraphics[width=1\linewidth]{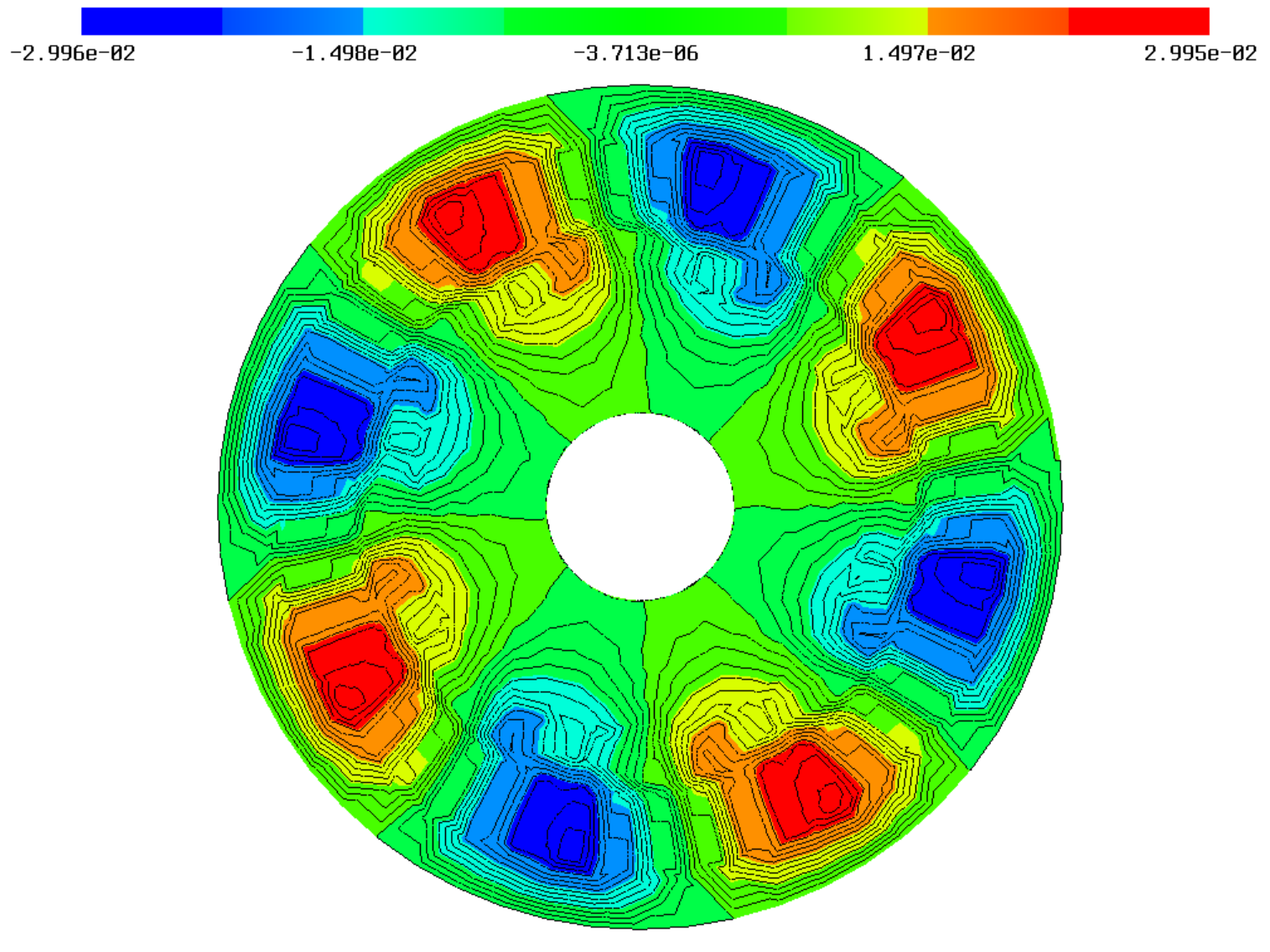} 
\textit{Solution at time $t = 0.009$.}
\end{center}
\end{minipage}
\hfill
\begin{minipage}{0.47\linewidth}
\begin{center}
\includegraphics[width=1\linewidth]{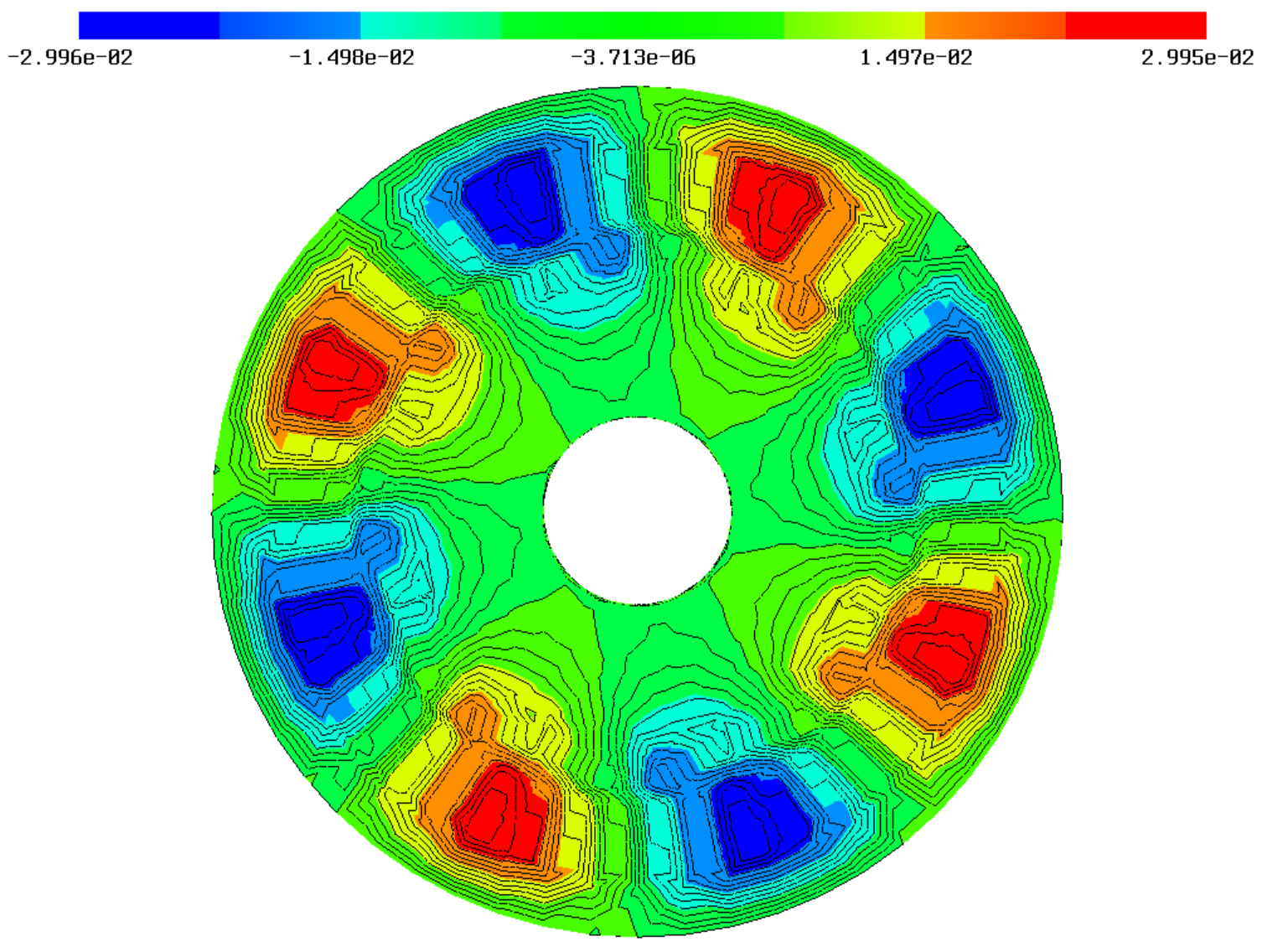} 
\textit{Solution at time $t = 0.015$.}
\end{center}
\end{minipage}
\caption{Cross sections of the solution for specific time points of the nonlinear problem with zero initial conditions.}
\label{fig:sol_motor_zero_initial_conditions}
\end{figure}

\begin{figure}[tbhp]
\begin{minipage}{0.47\linewidth}
\begin{center}
\includegraphics[width=1\linewidth]{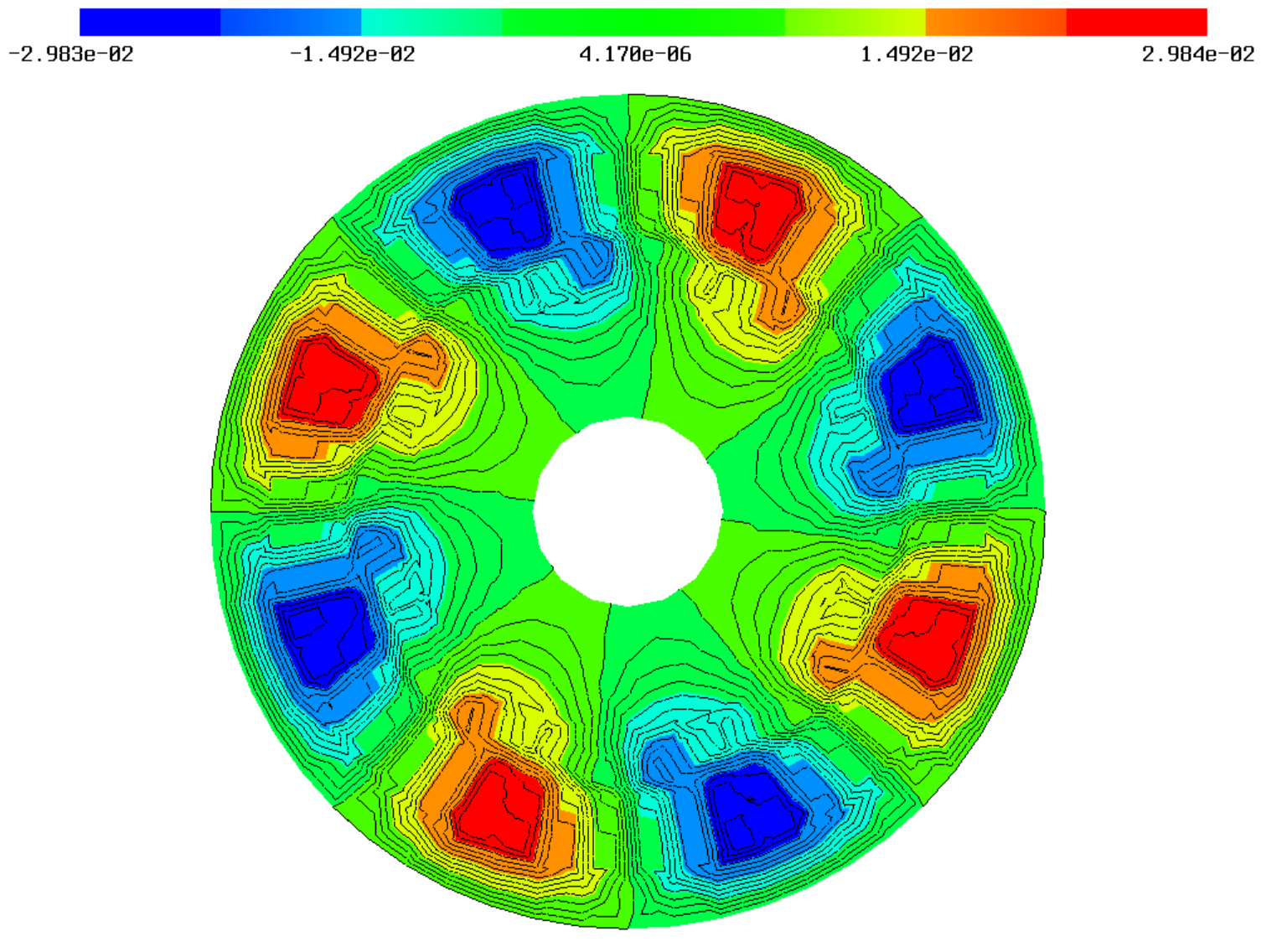}
\textit{Solution at time $t = 0.0$.}
\end{center} 
\end{minipage}
\hfill
\vspace{0.2 cm}
\begin{minipage}{0.47\linewidth}
\begin{center}
\includegraphics[width=1\linewidth]{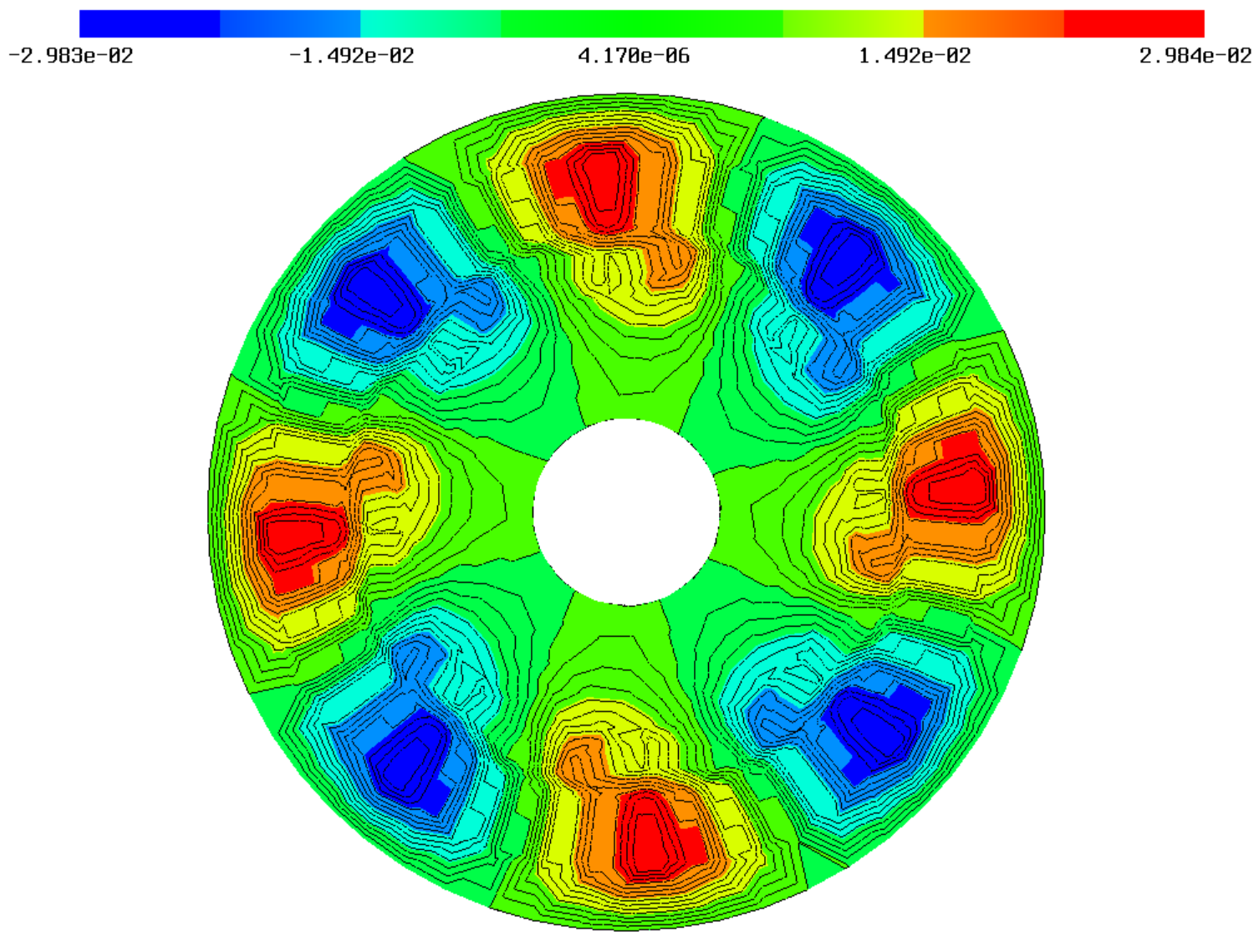} 
\textit{Solution at time $t = 0.0045$.}
\end{center}
\end{minipage}
\vfill
\vspace{0.2 cm}
\begin{minipage}{0.47\linewidth}
\begin{center}
\includegraphics[width=1\linewidth]{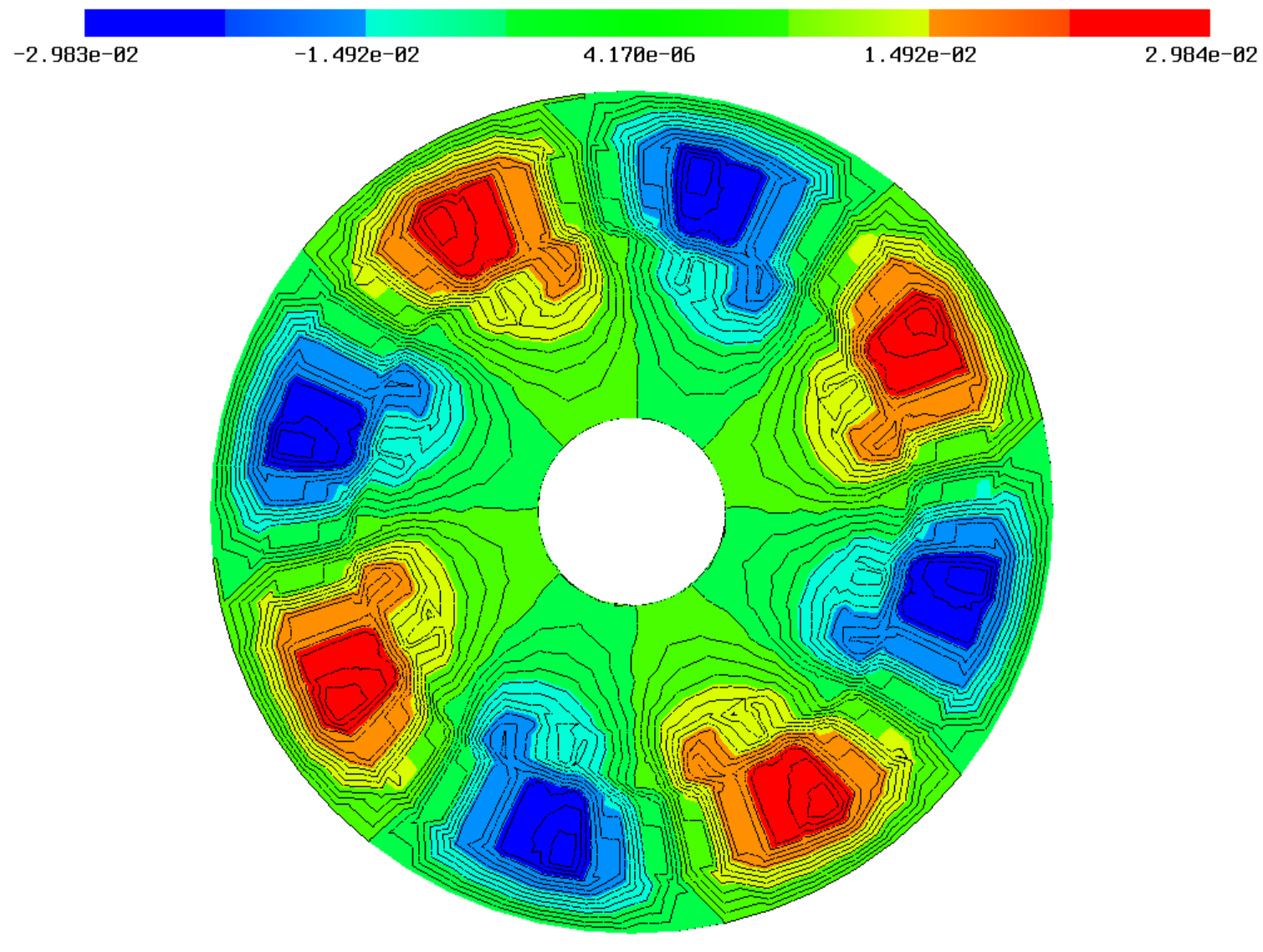} 
\textit{Solution at time $t = 0.009$.}
\end{center}
\end{minipage}
\hfill
\begin{minipage}{0.47\linewidth}
\begin{center}
\includegraphics[width=1\linewidth]{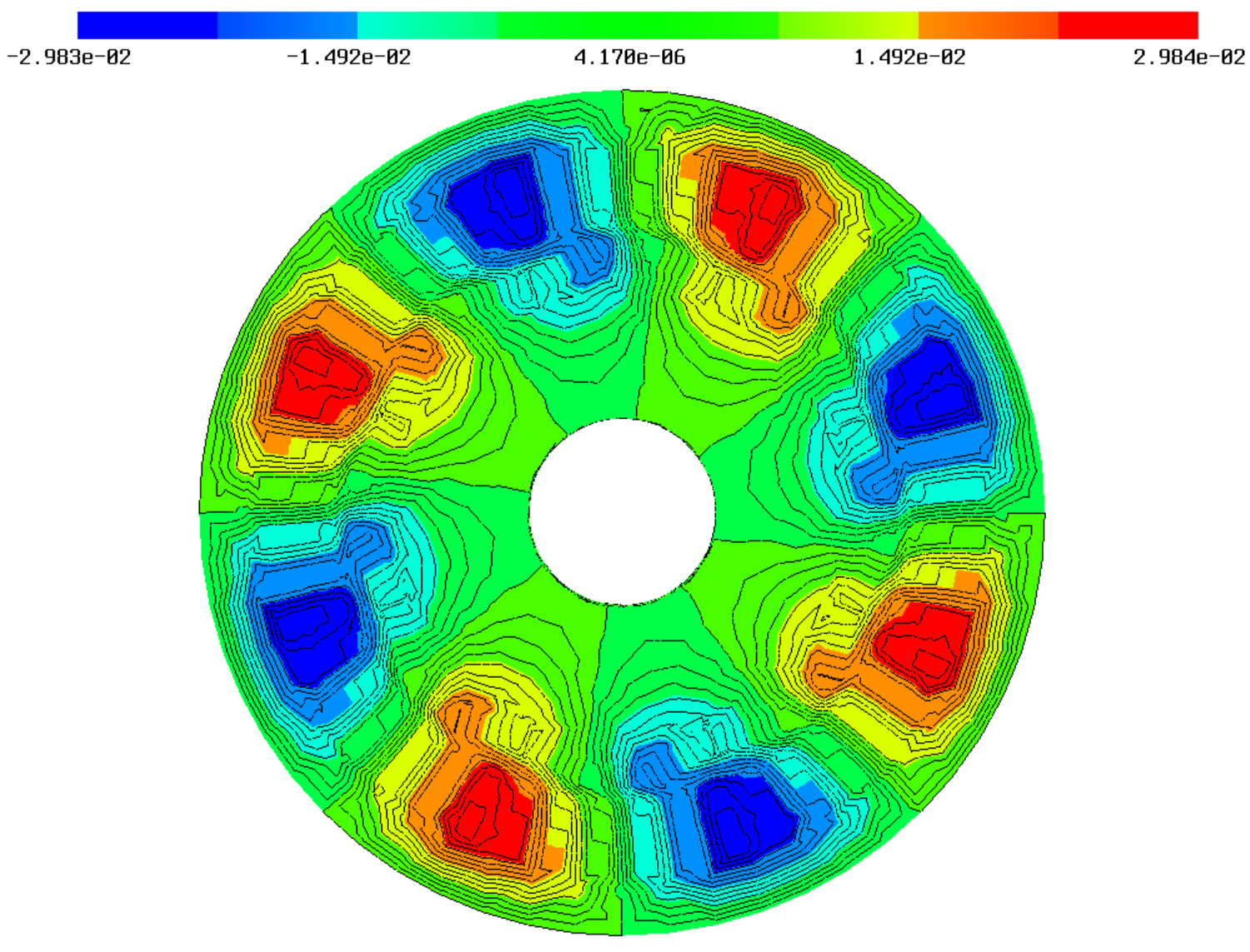}
\textit{Solution at time $t = 0.015$.}
\end{center}
\end{minipage}
\caption{Cross sections of the solution for specific time points of the nonlinear time periodic problem.}
\label{fig:sol_motor_periodic_conditions}
\end{figure}

Finally, we want to illustrate that our space-time method is also applicable to the magnetostatic problem which is obtained from \eqref{eqn:eddy_current_pde} by setting $\sigma = 0$ in the whole computational
domain. This yields a quasi-static problem where the right hand side and the geometry are time-dependent, but no time derivative of the solution is involved.
In this case, the underlying function spaces are
$X = Y = L^2(0,T;H^1_0(\Omega(t))$ 
with their corresponding conforming finite dimensional subspaces
$X_h = Y_h$ as described in Section~\ref{sec:space_time_fe_dis}.
We consider the nonlinear reluctivity $\nu_{\text{iron}}(|\nabla u|)$ and
 solve the resulting system in parallel
using again a damped version of Newton's method within the FE software Netgen/NGSolve \cite{netgen}.
In each Newton step the linearized system is solved with GMRES or MUMPS 
supported by PETSc \cite{DaPaKlCo2011}, where the computational times 
with respect to the number of cores are given in Table 
\ref{table:speedup_table_quasistatic_gmres} and Table 
\ref{table:speedup_table_quasistatic_mumps}, respectively. The solution is 
visualized by making cross sections at certain time points in Fig.~\ref{fig:sol_motor_quasistat_prob}.

\begin{table}
\caption{Computational times in seconds of the nonlinear magnetostatic problem solved with GMRES with 250 iterations in every Newton iteration with a maximum of 100 Newton iterations.}
\label{table:speedup_table_quasistatic_gmres}
\begin{center}
\begin{tabular}[tbhp]{ |c|c|c|c|c|c| } 
 \hline
 number of cores & 1 & 2 & 4 & 8 & 16\\
 \hline
 time in seconds & 2979 & 1618 & 897 & 490 & 292 \\ 
 \hline
\end{tabular}
\end{center}
\end{table}

\begin{table}
\caption{Computational times in seconds of the nonlinear magnetostatic problem
  solved with MUMPS in every Newton iteration within 53 Newton iterations.}
\label{table:speedup_table_quasistatic_mumps}
\begin{center}
\begin{tabular}[tbhp]{ |c|c|c|c|c|c| } 
 \hline
 number of cores & 1 & 2 & 4 & 8 & 16\\
 \hline
 time in seconds & 2166 & 1374 & 977 & 733 & 675 \\ 
 \hline
\end{tabular}
\end{center}
\end{table}

\begin{figure}[tbhp]
\begin{minipage}{0.47\linewidth}
\begin{center}
\includegraphics[width=1\linewidth]{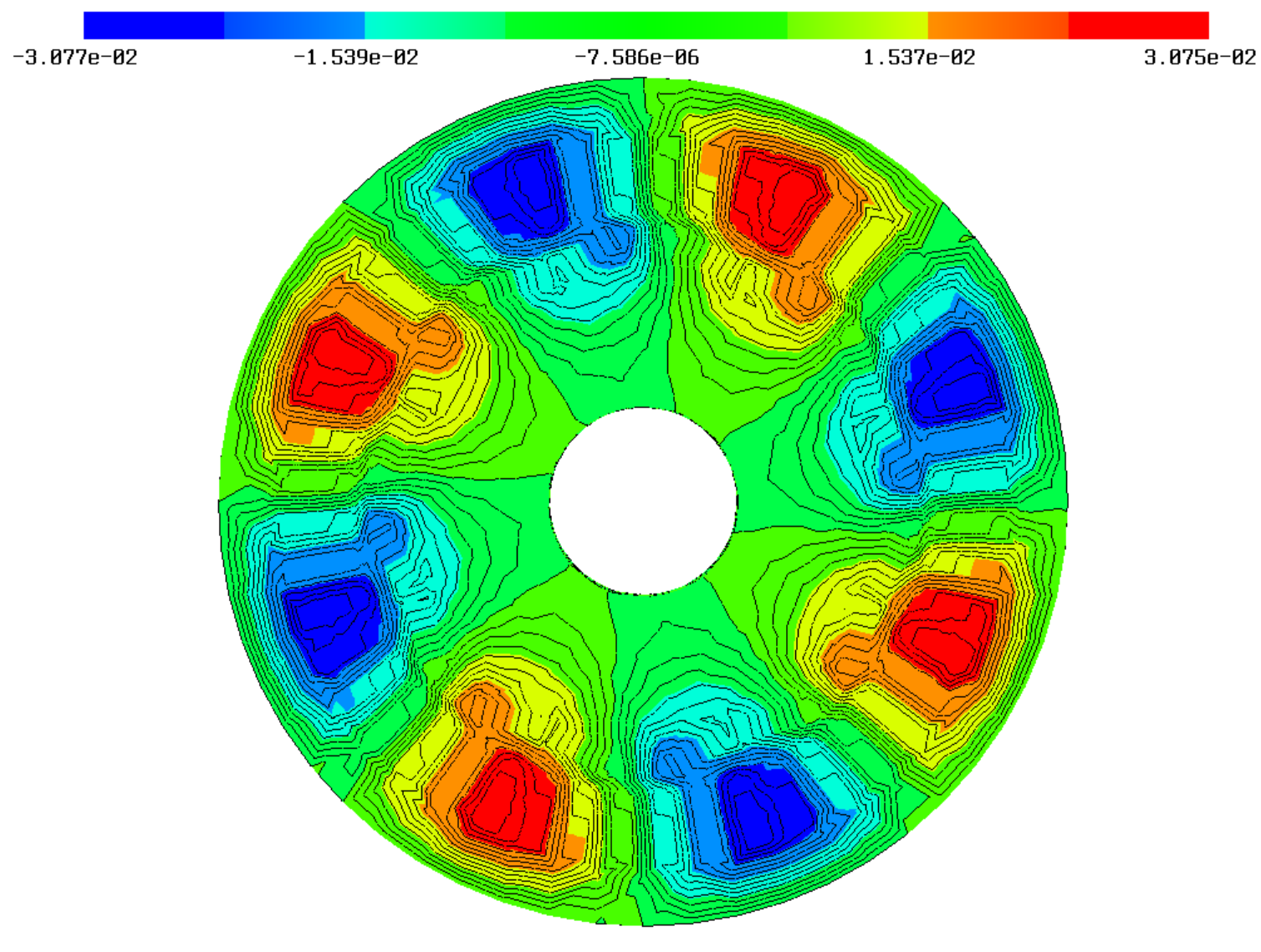}
\textit{Solution at time $t = 0.0$.}
\end{center} 
\end{minipage}
\hfill
\vspace{0.2 cm}
\begin{minipage}{0.47\linewidth}
\begin{center}
\includegraphics[width=1\linewidth]{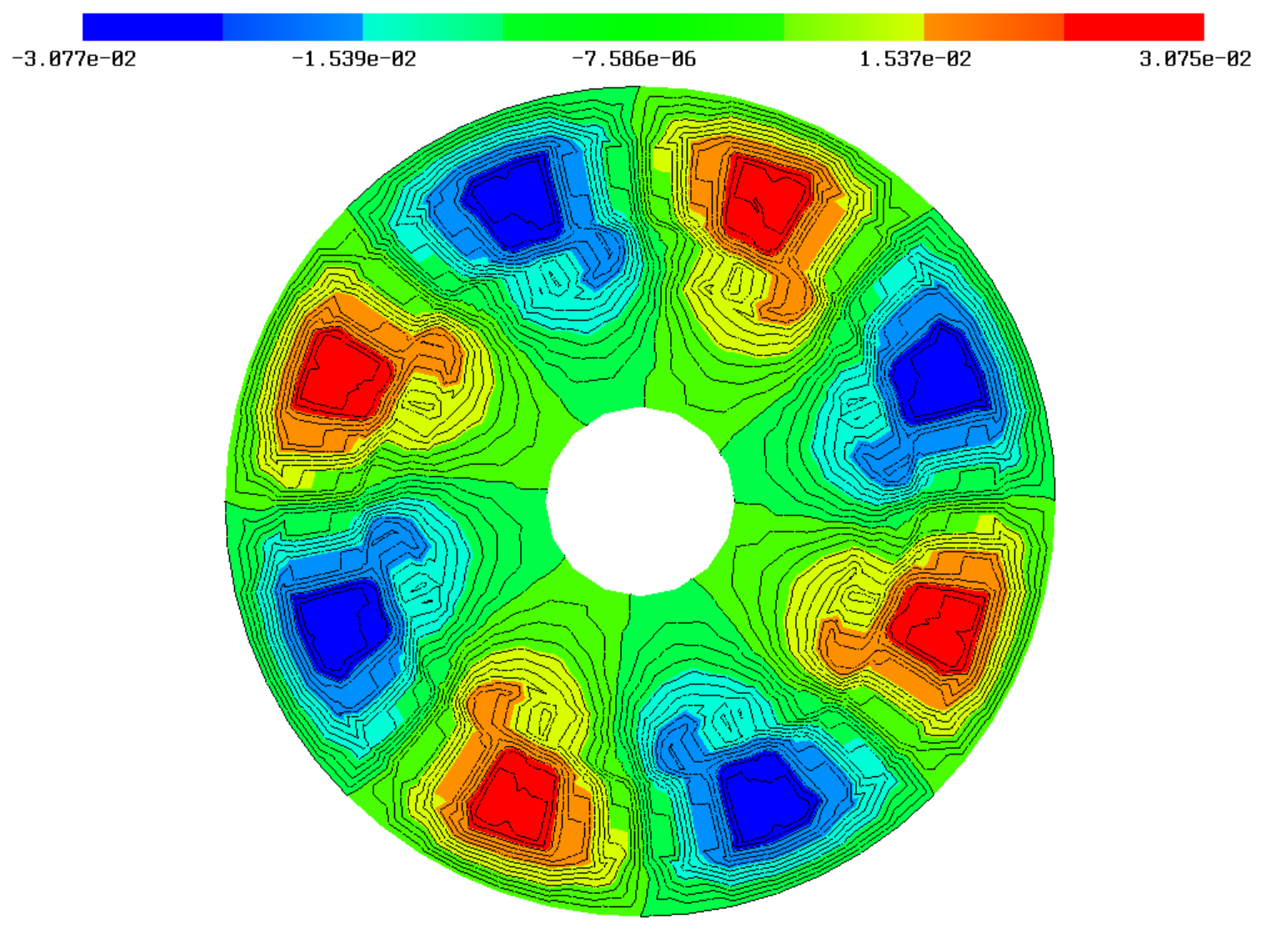} 
\textit{Solution at time $t = 0.0045$.}
\end{center}
\end{minipage}
\vfill
\vspace{0.2 cm}
\begin{minipage}{0.47\linewidth}
\begin{center}
\includegraphics[width=1\linewidth]{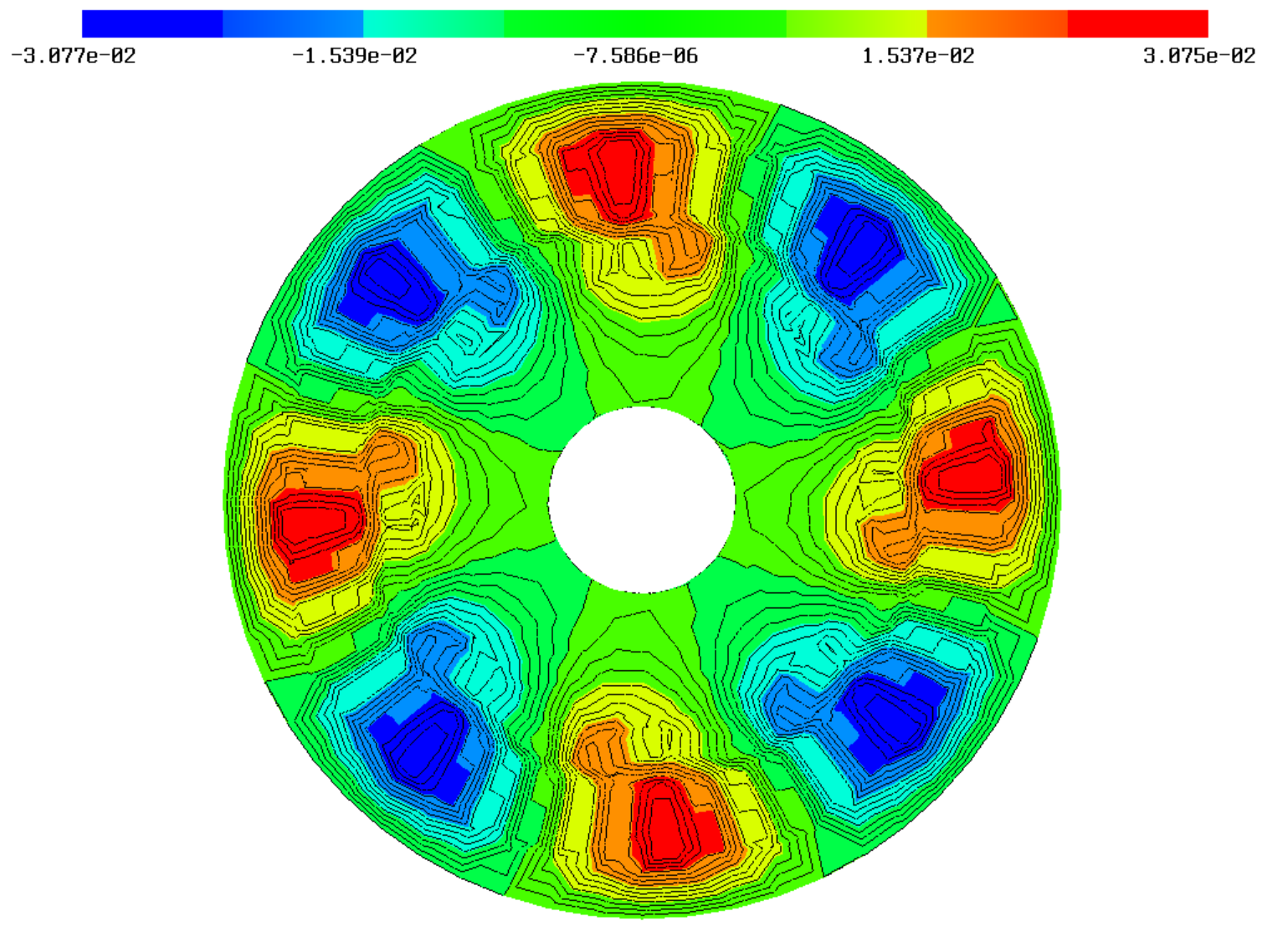} 
\textit{Solution at time $t = 0.009$.}
\end{center}
\end{minipage}
\hfill
\begin{minipage}{0.47\linewidth}
\begin{center}
\includegraphics[width=1\linewidth]{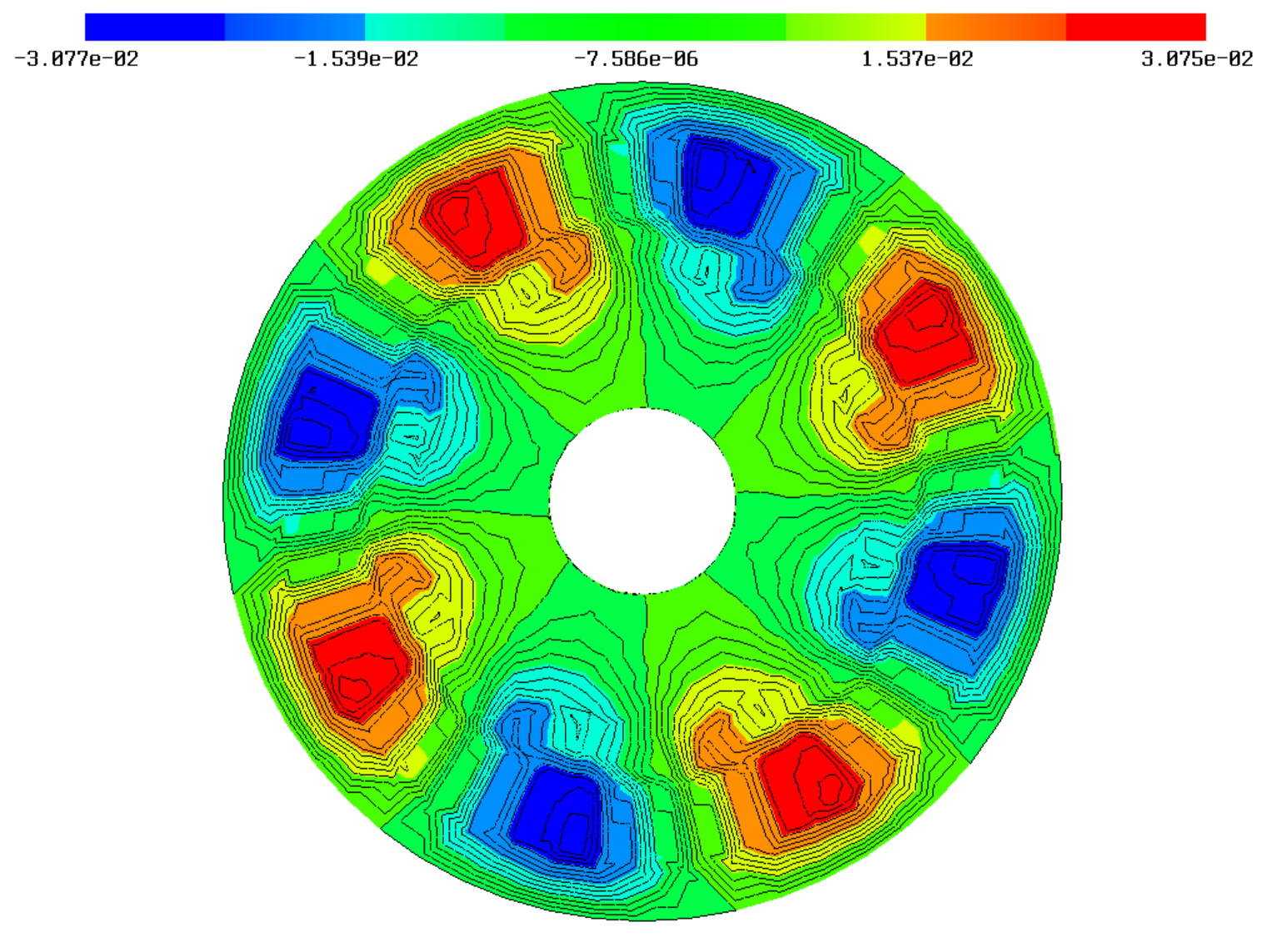}
\textit{Solution at time $t = 0.015$.}
\end{center}
\end{minipage}
\caption{Cross sections of the solution for specific time points of the nonlinear magnetostatic problem.}
\label{fig:sol_motor_quasistat_prob}
\end{figure}

\section{Conclusions}
\label{sec:conclusion}
In this paper we have formulated and analyzed a space-time finite
element method for the numerical simulation of electromagnetic fields
in rotating electric machines. As is the case of a fixed computational
domain we can apply the Babu\v{s}ka--Ne\v{c}as theory to
establish unique solvability. We have presented first numerical
results considering different settings for the mathematical model,
including a quasi-static model, as well as a nonlinear model to
describe the reluctivity. Although  we have applied this approach
already to an example of practical interest, it is still a challenging
task to improve the parallel solver in order to handle problems with
a much higher number of degrees of freedom. In addition to geometric or
algebraic multigrid methods we may use space-time domain
decomposition methods \cite{SteinbachGaulhofer} including
space-time tearing and interconnecting methods 
\cite{PachecoSteinbach}. 

\section*{Acknowledgments}
This work has been supported by the Austrian
Science Fund (FWF) under the Grant Collaborative Research Center
TRR361/F90: CREATOR Computational Electric Machine Laboratory.
P. Gangl acknowledges the support of the FWF project P~32911.
We would like to thank U.~Iben, J.~Fridrich, I.~Kulchytska-Ruchka,
O.~Rain, D.~Scharfenstein, and A.~Sichau
(Robert Bosch GmbH, Renningen, Germany) for the cooperation and
fruitful discussions during this work.

\bibliographystyle{abbrv}
\bibliography{references}

\end{document}